\documentclass[10pt,a4paper,oneside,reqno]{amsart}

\usepackage{setspace}
\usepackage[totalwidth=14cm,totalheight=21.5cm]{geometry}
\usepackage[T1]{fontenc}
\usepackage[dvips]{graphicx}
\usepackage{epsfig}
\usepackage{amsmath,amssymb,amscd,amsthm}
\usepackage{subfigure}
\usepackage[latin1]{inputenc}
\usepackage{latexsym}
\usepackage{graphics}
\usepackage{colortbl} 
\usepackage{color}
\usepackage{ae}
\usepackage{enumitem}
\usepackage[all]{xy}
\usepackage{scalerel}
\usepackage{tikz}
\usepackage{cancel}
\usepackage{bbm,amsbsy}
\usepackage{mathrsfs}  
\usepackage[colorlinks=true,pagebackref=true]{hyperref}
\usepackage[normalem]{ulem}
\usepackage{nicefrac}
\usepackage{xfrac}
\usepackage{faktor}
\usepackage{tikz-cd} 

\usepackage{nomencl}
\usepackage[normalem]{ulem}

\newtheorem{corollary}{Corollary}[section]

\newtheorem{theorem}[corollary]{Theorem}

\newtheorem{proposition}[corollary]{Proposition}

\newtheorem{lemma}[corollary]{Lemma}

\theoremstyle{definition}
\newtheorem{definition}[corollary]{Definition}
\theoremstyle{remark}
\newtheorem{remark}[corollary]{Remark}
\newtheorem*{remark*}{Remark}
\newtheorem{example}[corollary]{Example}

\newtheorem*{notation*}{Notation}

\newcommand{\C}{{\mathbb C}}

\newcommand{\N}{{\mathbb N}}
\newcommand{\R}{{\mathbb R}}

\newcommand{\Hyp}{\mathbb{H}}
\newcommand{\AdS}{\mathbb{A}\mathrm{d}\mathbb{S}}
\newcommand{\psl}{\mathfrak{sl}}
\newcommand{\SL}{\mathrm{SL}}
\newcommand{\PSL}{\mathrm{PSL}}

\newcommand{\Ker}{\mathrm{Ker}}

\newcommand{\tr}{\mathrm{tr}}
\newcommand{\gr}{\mathrm{graph}}
\newcommand{\mat}{\mathcal M_2(\R)}
\newcommand{\pmat}{\mathrm P\mat}
\newcommand{\con}{\mathcal C}

\newcommand{\id}{\mathrm{id}}

\newcommand{\RP}{\R \mathrm{P}}

\newcommand{\Fix}{\mathrm{Fix}}
\renewcommand{\1}{\mathbbm{1}}

\begin{document}

\setcounter{secnumdepth}{2}
\setcounter{tocdepth}{2}

\title[The Anti-de Sitter proof of Thurston's earthquake theorem]{The Anti-de Sitter proof of Thurston's earthquake theorem}

\author[Farid Diaf]{Farid Diaf}
\address{Farid Diaf: Univ. Grenoble Alpes, CNRS, IF, 38000 Grenoble, France.} \email{farid.diaf@univ-grenoble-alpes.fr}

\author[Andrea Seppi]{Andrea Seppi}
\address{Andrea Seppi: Univ. Grenoble Alpes, CNRS, IF, 38000 Grenoble, France.} \email{andrea.seppi@univ-grenoble-alpes.fr}

\thanks{The second author is member of the national research group GNSAGA}

\maketitle

\begin{abstract}
Thurston's earthquake theorem asserts that every orientation-preserving homeomorphism of the circle admits an extension to the hyperbolic plane which is a (left or right) earthquake. The purpose of these notes is to provide a proof of Thurston's earthquake theorem, using the bi-invariant geometry of the Lie group $\mathrm{PSL}(2,\mathbb R)$, which is also called Anti-de Sitter three-space. The involved techniques are elementary, and no background knowledge is assumed apart from some two-dimensional hyperbolic geometry.
\end{abstract}

\tableofcontents

\section{Introduction}

Since the 1980s, \emph{earthquake maps} have played an important role in the study of hyperbolic geometry and Teichm\"uller theory. These are (possibly discontinuous) maps of the hyperbolic plane to itself that, roughly speaking, are isometric in the complement of a subset of the hyperbolic plane which is a disjoint union of geodesics, and they "slip" along the "faults" represented by these geodesics. In particular, they may have points of discontinuity there. In general, an earthquake map can be complicated, and it is an isometry only on the connected components of the complement of a measured geodesic lamination.  

To achieve the solution of the Nielsen realization problem \cite{ker}, Steven Kerckhoff proved the so-called \emph{earthquake theorem} for closed hyperbolic surfaces, that is, the existence of a left (right) earthquake map between any two closed hyperbolic surfaces of the same genus. In \cite{thurston}, William Thurston gave a generalization, proved by independent methods, to a \emph{universal} setting, which is the statement that we consider in the present notes: he proved that every orientation-preserving homeomorphism of the circle admits an extension to the hyperbolic plane which is a (left or right) earthquake. Earthquake maps have  been extensively studied later in various directions, see \cite{ghl,hu,misaric,pfeil,saric,saric2,saric3}

\subsection*{Mess' groundbreaking work and later developments}

In his 1990 pioneering paper \cite{mess}, Geoffrey Mess has first highlighted the deep connections between the Teichm\"uller theory of hyperbolic surfaces, and three-dimensional Lorentzian geometries of constant sectional curvature. In particular, the so-called Anti-de Sitter geometry is the Lorentzian geometry of constant negative curvature --- that is, the Lorentzian analogue of hyperbolic geometry. One of the models of Anti-de Sitter three-space is simply the Lie group $\PSL(2,\R)$, endowed with a Lorentzian metric which is induced by the (bi-invariant) Killing form on its Lie algebra. This is the model that we adopt in the present work.

Mess has then observed that convex hulls in Anti-de Sitter space can be used, together with a \emph{Gauss map} construction for spacelike surfaces, to prove earthquake theorems in hyperbolic geometry. In \cite{mess}, Mess outlined the proof of the earthquake theorem between closed hyperbolic surfaces. His groundbreaking ideas have been improved and implemented by several authors, leading to many results of existence of earthquake maps in various settings \cite{bonschkra,bonsch,bonsch2,rosmondi} and of other interesting types of extensions \cite{bonsch3,bonsep,bonsep2,seppi}. See also the paper \cite{survey}, which is a detailed introduction to Anti-de Sitter geometry, contains a general treatment of the Gauss map, but only sketches some of the  ideas that appear in the proof of Mess.

The literature seems to lack a complete proof of the earthquake theorem, in Thurston's universal version, which relies on Anti-de Sitter geometry.  In this note, we will provide a detailed proof of Thurston's earthquake theorem (Theorem \ref{thm:main}), and we will then recover (Corollary \ref{cor:invariance}) the existence of earthquake maps for closed hyperbolic surfaces. While the proofs that appear in \cite{mess}, and in several of the aforementioned subsequent works, make use of a computation of the holonomy, here we will simply work with the definition of earthquake map. 

In fact, the proof presented here, although going through several technical steps, entirely involves elementary tools.
 The only required knowledge for these notes is the hyperbolic plane geometry in the upper half-space model, and the very basic definitions of Lie groups theory and Lorentzian geometry. 

\subsection*{A quick comparison of the two proofs}

It is also worth remarking that the proof presented here, and suggested by Mess, is not entirely different in spirit from Thurston's proof in \cite{thurston}. Indeed, the starting point of Thurston's proof consists in considering, given an orientation-preserving homeomorphism $f$ or the circle, those isometries $\gamma$ of  the hyperbolic plane such that the composition $h:=\gamma\circ f$ is \emph{extreme left}: that is, such that $h$ has a lift $\tilde h:\R\to\R$ satisfying $h(x)\leq x$ and whose fixed point set is non-empty. In Thurston's words,  "$h$ moves points counterclockwise on the circle, except for those points that it fixes". Then Thurston defines the earthquake map to be equal to $\gamma^{-1}$ on the convex hull of the fixed points of $h$. 

This has an interpretation in terms of Anti-de Sitter geometry. Spacelike planes in Anti-de Sitter space, which is simply the Lie group $\PSL(2,\R)$, are isometrically embedded copies of the hyperbolic plane, and are parameterized by elements of $\PSL(2,\R)$ itself, via a natural duality. For instance, the dual plane to the identity consists of all elliptic elements of order two, which is identified with the hyperbolic plane itself via the fixed point map. The ``extreme left condition'' as above is then exactly equivalent to the condition that the spacelike plane dual to $\gamma$ is a \emph{past support plane }of the convex hull of the graph of $f$, which can be seen as a subset of the boundary at infinity of Anti-de Sitter space.

The proof presented here then consists in considering the \emph{left and right projections}, defined on the past boundary components of the convex hull, and to consider the composition $E$ of one projection with the inverse of the other. It turns out that this composition map $E$ is indeed equal to $\gamma^{-1}$ on the convex hull of the fixed points of $\gamma\circ f$, as in Thurston's \emph{ansatz}. Of course one can replace extreme left by extreme right, and past boundary with future boundary, to obtain  right earthquakes instead of left earthquakes.

We remark that the main statement proved by Thurston also includes a uniqueness part. In fact, the earthquake map is not quite unique, but it is up to a certain choice that  has to be made at every geodesic where it is discontinuous. We will give an interpretation of this phenomenon in terms of a choice of support plane at the points of the boundary of the convex hull that admit several support planes, but we will not provide a proof of the uniqueness part here. 

\subsection*{Main elements of the Anti-de Sitter proof}

Despite the above analogies with Thurston's original proof of the existence of left and right earthquakes, developing the proof in the Anti-de Sitter setting then leads to remarkable differences with respect to Thurston's  proof. A large part of our proof is actually achieved by a reduction to the situation of an orientation-preserving homeomorphism of the circle which is equal to the restriction of an element $\gamma_i$ of $\PSL(2,\R)$ on an interval $I_i$ ($i=1,2$), where $I_1\cup I_2$ equals the circle. In this situation the earthquake extension is already well-known, and consists of a \emph{simple earthquake}. However, understanding this example in detail from the perspective of Anti-de Sitter geometry --- which corresponds to the situation where a boundary component of the convex envelope of $f$ is the union of two totally geodesic half-planes meeting along a geodesic --- then permits to prove easily some of the fundamental properties that one has to verify in order to show that the composition map $E$ is an earthquake map. 

There are furthermore two main technical statements that we have to prove. The first is the fact that the left and right projections (although they can be discontinuous) are bijective --- which is essential since the earthquake map is defined as the inverse of the left projection post-composed with the right projection, and implies that $E$ itself is a bijection of the hyperbolic plane. While injectivity is easy using the aforementioned example of two totally geodesic planes meeting along a geodesic, surjectivity requires a more technical argument. The second statement is an extension lemma, which ensures that the left and right projections (although sometimes discontinuous) extend continuously to the boundary, and the extension is simply  the projection from the graph of $f$ onto the first and second factor. This ensures that the composition $E$ of the right projection with the inverse of the left projection extends to $f$ itself on the circle at infinity.

Some of the above steps do of course involve a number of technical difficulties, but the language of Anti-de Sitter geometry is, in our opinion, extremely effective, and permits to stick to quite elementary techniques in the entire work. 

\subsection*{Acknowledgements}

We would like to thank Pierre Will for a remark on the description of timelike planes via composition of orientation-reversing isometries, that is used in Section \ref{sec:timelike planes}. We are grateful to Filippo Mazzoli and Athanase Papadopoulos   for useful suggestions that helped improving the exposition.

\section{Earthquake maps}
\label{sec:eart}

Throughout this work, we will use the upper half-plane model of the hyperbolic plane $\Hyp^2$, that is, $\Hyp^2$ is the half-space $\mathrm{Im}(z)>0$ in $\C$ endowed with the Riemannian metric $|dz|^2/\mathrm{Im}(z)^2$ of constant curvature $-1$. Its visual boundary $\partial_\infty\Hyp^2$ is therefore identified with $\R\cup\{\infty\}$, and $\overline{\Hyp}{}^2=\Hyp^2\cup\partial_\infty\Hyp^2$ is endowed with the topology given by the one-point compactification of the closed half-plane $\mathrm{Im}(z)\geq 0$. The isometry group of $\Hyp^2$ is identified with the group $\PSL(2,\R)$ acting by homographies, and its action naturally extends to $\partial_\infty\Hyp^2$.

\begin{definition}
A geodesic lamination $\lambda$ of $\Hyp^2$ is a collection of
disjoint geodesics that foliate a closed subset $X\subseteq\Hyp^2$. The closed set $X$ is called the \emph{support} of $\lambda$. The geodesics in $\lambda$ are called \emph{leaves}. The connected components of the complement $\mathbb{H}^2\setminus X$ are called \emph{gaps}. The \emph{strata} of $\lambda$ are the leaves and the gaps.
\end{definition}

Given a hyperbolic isometry $\gamma$ of $\Hyp^2$, the \emph{axis} of $\gamma$ is the geodesic $\ell$ of $\Hyp^2$ connecting the two fixed points of $\gamma$ in $\partial_\infty\Hyp^2$. Therefore the axis $\ell$ is preserved by $\gamma$, and when restricted to $\ell$, $\gamma|_\ell:\ell\to\ell$ acts as a translation  with respect to any constant speed parameterization of $\ell$. 

Given two subsets $A,B$ of $\Hyp^2$, we say that a geodesic $\ell$ \emph{weakly separates} $A$ and $B$ if $A$ and $B$ are contained in the closure of different connected components of $\Hyp^2\setminus \ell$.

\begin{definition}\label{defi:eart}
A \emph{left} (resp. \emph{right}) \emph{earthquake} of $\Hyp^2$
is a bijective map $E:\mathbb{H}^2\to \mathbb{H}^2$ such that there exists a  geodesic lamination $\lambda$ for which
the restriction $E|_{S}$ of $E$ to any stratum $S$ of $\lambda$ is equal to the restriction of
an isometry of $\Hyp^2$, and 
for any two strata $S$ and $S'$ of $\lambda$, the comparison isometry
$$\mathrm{Comp}(S,S'):=(E|_{S})^{-1}\circ E|_{S'}$$ is the restriction of an isometry $\gamma$ of $\Hyp^2$, such that:
\begin{itemize}
\item $\gamma$ is different from the identity, unless possibly when one of the two strata $S$ and $S'$ is contained in the closure of the other;
\item 
when it is not the identity, 
 $\gamma$ is a hyperbolic transformation whose axis $\ell$ weakly separates $S$ and $S'$;
\item moreover, $\gamma$
translates to the left (resp. right), seen from $S$ to $S'$. 
\end{itemize}
\end{definition}

Let us clarify the meaning of this last condition. Suppose $f:[0,1]\to\Hyp^2$ is smooth path  such that $f(0)\in S$, $f(1)\in S'$ and the image of $f$ intersects $\ell$ transversely and exactly at one point $z_0=f(t_0)\in\ell$. Let $v=f'(t_0)\in T_{z_0}\Hyp^2$ be the tangent vector at the intersection point. Let $w\in T_{z_0}\Hyp^2$ be a vector tangent to the geodesic $\ell$  pointing towards $\gamma(z_0)$. Then we say that $\gamma$ translates to the left (resp. right) seen from $S$ to $S'$ if $v,w$ is a positive (resp. negative) basis of $T_{z_0}\Hyp^2$, for the standard orientation of $\Hyp^2$.

It is important to observe that this condition is independent of the order in which we choose $S$ and $S'$. That is, if $\mathrm{Comp}(S,S')$ translates to the left (resp. right) seen from $S$ to $S'$, then $\mathrm{Comp}(S',S)$ translates to the left (resp. right) seen from $S'$ to $S$.

We remark that  an earthquake $E$ is not required to be continuous. In fact, in some cases it will \emph{not} be continuous, for instance when the lamination $\lambda$ is finite, meaning that $\lambda$ is a collection of a finite number of geodesics. This is best visualized in the following simple example.

\begin{example}\label{ex:simple}
The map
$$E:\Hyp^2\to\Hyp^2$$
defined in the upper half-space model of $\Hyp^2$ by:
$$E(z)=\begin{cases}
z & \text{if }\mathrm{Re}(z)<0 \\
a z & \text{if }\mathrm{Re}(z)=0 \\
b z & \text{if }\mathrm{Re}(z)>0
\end{cases}$$
is a left earthquake if $1<a<b$, and a right earthquake if $0<b<a<1$. The lamination $\lambda$ that satisfies Definition \ref{defi:eart} is composed of a unique geodesic, namely the geodesic $\ell$ with endpoints $0$ and $\infty$. 
\end{example}

It is clear that  the earthquake map $E$ from Example \ref{ex:simple} is not continuous along $\ell$. 

Despite the lack of continuity, Thurston proved that any earthquake map extends continuously to an orientation-preserving homeomorphism of $\partial_\infty\Hyp^2$, meaning that there exists a (unique) orientation-preserving homeomorphism $f:\partial_\infty\Hyp^2\to\partial_\infty\Hyp^2$ such that the map 
$$\overline E(z)=\begin{cases}
E(z) &\text{if }z\in\Hyp^2 \\
f(z) &\text{if }z\in\partial_\infty\Hyp^2
\end{cases}$$
is continuous at every point of $\partial_\infty\Hyp^2$.

Then Thurston provided a proof of the following theorem, that he called \emph{``geology is transitive''}:

\begin{theorem}\label{thm:main}
Given any orientation-preserving homeomorphism $f:\partial_\infty\Hyp^2\to\partial_\infty\Hyp^2$, there exists a left earthquake map of $\Hyp^2$, and a right earthquake map, that extend continuously to $f$ on $\partial_\infty\Hyp^2$.
\end{theorem}

We remark that the earthquake map is not unique, as shown by Example \ref{ex:simple}, which provides a family of left (resp. right) earthquake maps extending the homeomorphism 
$$
f(x)=\begin{cases}
x & \text{if }x\leq 0 \\
b x & \text{if }x\geq 0 \\
\infty & \text{if } x=\infty
\end{cases}~,
$$
parameterized by the choice of $a\in(1,b)$ (resp. $a\in (b,1)$). Thurston's theorem is actually stronger than the statement of Theorem \ref{thm:main} above, since it characterizes the non-uniqueness as well. In short, the range of choices of the earthquake extension as in Example \ref{ex:simple} is essentially the only indeterminacy that occurs, and it happens exactly on each leaf of the lamination where the earthquake is discontinuous. We will not deal with the uniqueness part as in Thurston's work here. Nevertheless, in Subsection \ref{subsec:closed} we will show that our proof permits to recover the existence of earthquake maps between homeomorphic closed hyperbolic surfaces, not relying on the uniqueness property.

\section{Anti-de Sitter geometry}
\label{sec:ads}

In this section, we will introduce the fundamental notions in Anti-de Sitter geometry. For more details, the reader can consult \cite[Section 3]{survey}.

\subsection{First definitions}
The three-dimensional Anti-de Sitter space $\AdS^3$ is the Lie group $\PSL(2,\R)$, that is, the group of orientation-preserving isometries of $\Hyp^2$, endowed with a bi-invariant metric of signature $(2,1)$ (namely, a Lorentzian metric) which we now construct. 

Consider first the double cover $\SL(2,\R)$ of $\PSL(2,\R)$, which we realize as the subset of matrices of unit determinant in the four-dimensional vector space $\mat$ of 2-by-2 matrices. Endow $\mat$ with the quadratic form $q$:
$$q(A)=-\det (A)~.$$
It can be checked that $q$ has signature $(2,2)$. The associated bilinear form is expressed by the formula:
\begin{equation}\label{eq:product}
\langle A,B\rangle=-\frac{1}{2}\tr(A\cdot \mathrm{adj}(B))
\end{equation}
for $A,B\in \mat$, where $\mathrm{adj}$ denotes the adjugate matrix, namely
$$\mathrm{adj}\begin{pmatrix} a & b \\ c & d\end{pmatrix}=\begin{pmatrix} d & -b \\ -c & a\end{pmatrix}~.$$

Then $\SL(2,\R)$ is realized as the subset of $\mat$ defined by the condition $q(A)=-1$, and the restriction of $\langle\cdot,\cdot\rangle$ to the tangent space of $\SL(2,\R)$ at every point defines a pseudo-Riemannian metric of signature $(2,1)$. We will still denote by $\langle\cdot,\cdot\rangle$ this metric on $\SL(2,\R)$, and by $q$ the corresponding quadratic form. It can be shown that this metric has constant curvature $-1$, and the restriction of $\langle\cdot,\cdot\rangle$ to the Lie algebra $\psl(2,\R)$ coincides with $1/8$ times the Killing form of $\SL(2,\R)$.

Clearly both $\SL(2,\R)$ and $q$ are invariant under multiplication by minus the identity matrix, hence the quotient $\PSL(2,\R)=\SL(2,\R)/\{\pm 1\}$ is endowed with a Lorentzian metric of constant curvature $-1$, and is what we call the (three-dimensional) Anti-de Sitter space $\AdS^3$. It turns out that the group of orientation-preserving and time-preserving isometries of $\AdS^3$ is the group $\PSL(2,\R)\times\PSL(2,\R)$, acting by left and right multiplication on $\PSL(2,\R)\cong\AdS^3$:

$$(\alpha,\beta)\cdot \gamma:=\alpha\gamma\beta^{-1}~.$$

Although orientation-preserving and time-preserving are notions that do not depend on a chosen orientation, we will fix here an orientation and a time-orientation of $\AdS^3\cong\PSL(2,\R)$. To define an orientation on a Lie group, it actually suffices to define it in the Lie algebra, namely the tangent space at the identity $\1$. Hence we declare that the following is an oriented basis (which is actually orthonormal) of $\psl(2,\R)$:

\begin{equation} \label{eq:orientation}
V= \begin{pmatrix}
0 & 1 \\
1 & 0 
\end{pmatrix} \qquad W=\begin{pmatrix}
1 & 0 \\
0 & -1 
\end{pmatrix} \qquad U=\begin{pmatrix}
0 & -1 \\
1 & 0 
\end{pmatrix}
\end{equation}
Observe that the vectors $V,W$ are spacelike (i.e. $q(V,V)>0$ and $q(W,W)>0$), while $U$ is timelike ($q(U,U)<0$). One can check that $U$ is the tangent vector to the one-parameter group of elliptic isometries of $\Hyp^2$ fixing $i\in\Hyp^2$, parameterized by the angle of clockwise rotation; $V$ and $W$ are vectors tangent to the one-parameter groups of hyperbolic isometries fixing the geodesics with endpoints $(-1,1)$ and $(0,\infty)$ respectively.
Analogously, to define a time-orientation it suffices to define it in the Lie algebra, and we declare that $U$ is a future-pointing timelike vector.

\subsection{Boundary at infinity}

The boundary at infinity of $\AdS^3$ is defined as the projectivization of the cone of rank one matrices in $\mat$:
$$\partial_\infty\AdS^3=\mathrm P\{A\,|\,q(A)=0,\,A\neq 0\}~.$$
We endow $\overline\AdS{}^3=\AdS^3\cup\partial_\infty\AdS^3$ with the topology induced by seeing both $\AdS^3$ and $\partial_\infty\AdS^3$ as subsets of the (real) projective space $\pmat$ over the vector space $\mat$. Hence $\overline\AdS{}^3$ is the compactification of $\AdS^3$ in $\pmat$. It will be extremely useful to consider the homeomorphism between $\partial_\infty\AdS^3$ and $\RP^1\times\RP^1$, which is defined as follows:
\begin{equation}\label{eq:bdy}
\begin{array}{ccccc}
\delta & : & \partial_{\infty} \AdS^3 & \to & \RP^1\times\RP^1 \\
 & & [X] & \mapsto & (\mathrm{Im}(X),\mathrm{Ker}(X)) \\
\end{array}
\end{equation}
where of course in the right-hand side we interpret $\RP^1$ as the space of one-dimensional subspaces of $\R^2$. Since we have $\mathrm{Im}(AXB^{-1})=A\cdot \mathrm{Im}(X)$ and $\mathrm{Ker}(AXB^{-1})=B\cdot \mathrm{Ker}(X)$, the map $\delta$ is equivariant with respect to the action of the  group $\PSL(2,\R)\times\PSL(2,\R)$, acting on $\partial_\infty\AdS^3$ as the natural extension of the group of isometries of $\AdS^3$, and on $\RP^1\times\RP^1$ by the obvious product action.

The following is a useful characterization of  sequences in $\AdS^3$ converging to a point in the boundary (see \cite[Lemma 3.2.2]{survey}): for $(\gamma_n)_{n\in \N}$ a sequence of isometries of $\Hyp^2$, we have:
\begin{equation}\label{eq:convergence}
\begin{split}
\gamma_n\to \delta^{-1}(x,y)&\Leftrightarrow \text{there exists }z\in\Hyp^2\text{ such that }\gamma_n(z)\to x\text{ and }\gamma_n^{-1}(z)\to y  \\
& \Leftrightarrow  \text{for every }z\in\Hyp^2,\,\gamma_n(z)\to x\text{ and }\gamma_n^{-1}(z)\to y
\end{split}
\end{equation}
where of course here we are using the standard identification between $\RP^1$ and the visual boundary $\R\cup\{\infty\}=\partial_\infty\Hyp^2$, mapping the line spanned by $(a,b)$ to $a/b$. 

A fundamental step in the proof of the earthquake theorem is that to any map $f:\partial_\infty\Hyp^2\to \partial_\infty\Hyp^2$ we can associate a subset of $\partial_\infty\AdS^3$, namely (via the map $\delta$) the graph of $f$. By the equivariance of the map $\delta$ introduced in \eqref{eq:bdy}, we see immediately that, for $(\alpha,\beta)\in \PSL(2,\R)\times\PSL(2,\R)$:

\begin{equation}\label{eq:action graph}
(\alpha,\beta)\cdot \mathrm{graph}(f)=\mathrm{graph}(\beta f\alpha^{-1})~.
\end{equation}

In the rest of this paper, we will omit the map $\delta$, and we will simply identify $\partial_\infty\AdS^3$ with $\RP^1\times\RP^1$.

\subsection{Spacelike planes}

We conclude the preliminaries by an analysis of totally geodesic planes in $\AdS^3$. They are all obtained as the intersection of $\AdS^3$ with a projective subspace in the projective space $\pmat$ over $\mat$. Hence they are all of the following form:
\begin{equation}\label{eq:plane}
P_{[A]}=\{[X]\in\PSL(2,\R)\,|\,\langle X,A\rangle=0\}
\end{equation}
for some nonzero 2-by-2 matrix $A$. The notation $P_{[A]}$ is justified by the observation that the plane $P_A$ defined in the right-hand side of   \eqref{eq:plane}  depends only on the projective class of $A$. The totally geodesic plane $P_{[A]}$ is spacelike (resp. timelike, lightlike) if and only if $q(A)=-\det(A)$ is negative (resp. positive, null).   It will be called the \emph{dual plane} of $[A]$, since it can be seen as a particular case of the usual projective duality between points and planes in projective space. In particular, the dual plane $P_\gamma$ of an element $\gamma\in \PSL(2,\R)$ is a spacelike totally geodesic plane.

\begin{example}\label{ex:traceless}
The first example, which is of fundamental importance for the following, is for $\gamma=\1$ is the identity of $\PSL(2,\R)$. By \eqref{eq:product}, $P_\1$ is the subset of $\PSL(2,\R)$ consisting of projective classes of unit matrices $X$ with $\tr(X)=0$. By the Cayley--Hamilton theorem, $X^2=-\id$, hence the elements of $P_\1$ are order--two isometries of $\Hyp^2$, that is, elliptic elements with rotation angle $\pi$. Observe that $P_\1$ is invariant under the action of $\PSL(2,\R)$ by conjugation, which corresponds to the diagonal in the isometry group $\PSL(2,\R)\times\PSL(2,\R)$ of $\AdS^3$. Using \eqref{eq:convergence}, one immediately sees that the boundary of $P_\1$ in $\partial_\infty\AdS^3\cong \RP^1\times\RP^1$ is the diagonal; more precisely:
\begin{equation}\label{eq:graph id}
\partial_\infty P_\1=\mathrm{graph}(\1)\subset \RP^1\times\RP^1~.
\end{equation}

Given a point $z\in\Hyp^2$, let us denote by $\mathcal R_z$ the order--two elliptic isometry with fixed point $z$. We claim that the map 
$$\iota:\Hyp^2\to P_\1\qquad \iota(z)=\mathcal R_z$$
is an isometry with respect to the hyperbolic metric of $\Hyp^2$ and the induced metric on $P_\1\subset\AdS^3$. First, the inverse of $\iota$ is simply the fixed-point map $\Fix:P_\1\to\Hyp^2$ sending an elliptic isometry to its fixed point, which also shows that $\iota$ is equivariant with respect to the action of $\PSL(2,\R)$ on $\Hyp^2$ by homographies and on $P_\1$ by conjugation, since $\Fix(\alpha\gamma\alpha^{-1})=\alpha(\Fix(\gamma)$). That is, we have the relation
\begin{equation}\label{eq:uivariance}
\iota(\alpha\cdot p)=\alpha\circ\iota(p)\circ\alpha^{-1}~.
\end{equation}
This immediately implies that $\iota$ is isometric, since the pull-back of the metric of $P_\1$ is necessarily $\PSL(2,\R)$-invariant and has constant curvature $-1$, hence it coincides with the standard hyperbolic metric on the upper half-space.
\end{example}

This example is actually the essential example to understand general spacelike totally geodesic planes. Indeed, every \emph{spacelike} totally geodesic  plane is of the form $P_\gamma$ for some $\gamma\in \PSL(2,\R)$. To see this, observe that the action of the isometry group of $\AdS^3$ on spacelike totally geodesic  planes is transitive, and that $P_\gamma=(\gamma,\1)P_{\1}$ because the isometry $(\gamma,\1)$ maps $\1$ to $\gamma$, and therefore maps the dual plane of $\1$ to the dual plane of $\gamma$. By \eqref{eq:action graph} and \eqref{eq:graph id}, we immediately conclude the following:

\begin{lemma}\label{lem:spacelike planes}
Every spacelike totally geodesic plane of $\AdS^3$ is of the form $P_{\gamma}$ for some orientation-preserving isometry $\gamma$ of $\Hyp^2$, and
$$\partial_\infty P_\gamma=\mathrm{graph}(\gamma^{-1})\subset \RP^1\times\RP^1~.$$
\end{lemma}


\subsection{Timelike planes}\label{sec:timelike planes}

Let us now consider a matrix $A\in\mat$ such that $\det(A)=-1$. Hence the plane defined by Equation \eqref{eq:plane} is a timelike totally geodesic plane. Associated with $[A]$ is an orientation-reversing isometry $\eta$ of $\Hyp^2$ . Indeed, the action of $A$ by homography on $\mathbb C\mathrm P^1$ preserves $\RP^1$ and switches the two connected components of the complement, that is, the upper and the lower half-spaces. The matrix $A$ thus induces an orientation-reversing isometry,  up to identifying these two components via $z\mapsto \bar z$. We will thus denote $P_{[A]}$ by $P_{\eta}$, by a small abuse of notation.

The totally geodesic plane $P_{\eta}$ can be parameterized as follows. Consider the map 
\begin{equation}\label{eq:reflections}
\mathcal I\mapsto  \mathcal{I} \eta~,
\end{equation} defined on the space of reflections $\mathcal I$ along geodesics of $\Hyp^2$, with values in $\PSL(2,\R)\cong\AdS^3$. Its image is precisely $P_{\eta}$. Indeed, it is useful to remark that by the Cayley--Hamilton theorem, a matrix $X$ with $\det(X)=-1$ is an involution if and only if  and $\tr(X)=0$. Now, because $\det(A)=-1$, $\mathrm{adj}(A)=-A^{-1}$, and therefore $\langle XA,A\rangle=0$ if and only if $\tr(X)=0$, that is, if and only if $X$ is an involution. This shows that the image of the map \eqref{eq:reflections} is the entire plane $P_{\eta}$.

Similarly to the spacelike case, using the transitivity of the action of the group of isometries on timelike planes, every timelike plane is of the form above. Thanks to this description, we can show the following.

\begin{lemma}\label{lem:timelike planes}
Every timelike totally geodesic plane of $\AdS^3$ is of the form $P_{\eta}$ for some orientation-reversing isometry $\eta$ of $\Hyp^2$, and
$$\partial_\infty P_\eta=\mathrm{graph}(\eta^{-1})\subset \RP^1\times\RP^1~.$$
\end{lemma}
\begin{proof}
It only remains to check the identity for $\partial_\infty P_\eta$. For this, we will use the characterization \eqref{eq:convergence} together with the parameterization \eqref{eq:reflections} of $P_\eta$. Suppose the sequence $\mathcal I_n$ is such that $\mathcal I_n\eta(z)\to x\in\partial_\infty\Hyp^2$, for any $z\in\Hyp^2$. Then, using that $\mathcal I_n$ is an involution and the continuity of the action of $\eta$ on $\overline\Hyp{}^2$, $(\mathcal I_n\eta)^{-1}(z)=\eta^{-1}\mathcal I_n^{-1}(z)=\eta^{-1}\mathcal I_n(z)\to\eta^{-1}(x)$. This concludes the proof.
\end{proof}

\begin{remark}
It is worth remarking that, since reflections of $\Hyp^2$ are uniquely determined by (unoriented) geodesics, we can consider the map \eqref{eq:reflections} as a map from the space $\mathcal G(\Hyp^2)$ of unoriented geodesics of $\Hyp^2$ to $\PSL(2,\R)$. It turns out that this map is isometric with respect to a natural metric on $\mathcal G(\Hyp^2)$ which makes it identified with the two-dimensional Anti-de Sitter space $\AdS^2$, see \cite[Example 6.1]{elemamseppi} for more details. 
\end{remark}

\subsection{Lightlike planes}

The only case  left to consider  consists of lightlike totally geodesic planes. Those are of the form $P_{[A]}$ for a nonzero matrix $A$ with $\det(A)=0$, that is, for $\mathrm{rank}(A)=1$. We describe their boundary in the following lemma. It is important to remark that, unlike spacelike and timelike planes considered above, the boundary will not be a graph in $\RP^1\times\RP^1$.
\begin{lemma}\label{lem:lightlike planes}
Every lightlike totally geodesic plane of $\AdS^3$ is of the form $P_{[A]}$ for some rank one matrix $A$, and
$$\partial_\infty P_{[A]}=\left(\mathrm{Im}(A)\times\RP^1\right)\cup\left(\RP^1\times\Ker(A)\right)~.$$
\end{lemma}
In other words, $\partial_\infty P_{[A]}$ is the union of two circles in $\RP^1\times\RP^1$, one horizontal and one vertical, which intersect exactly at the point in $\RP^1\times\RP^1$ corresponding to $[A]\in\partial_\infty\AdS^3$ via the map $\delta$ introduced in \eqref{eq:bdy}. 
\begin{proof}
The points in $\partial_\infty P_{[A]}$ are projective classes of rank one matrices $X$ satisfying $\langle X,A\rangle=0$, that is, such that $\tr(X\mathrm{adj}(A))=0$. Since $X\mathrm{adj}(A)$ has vanishing determinant, by the Cayley-Hamilton theorem $X\mathrm{adj}(A)$ is traceless if and only if it is nilpotent, that is, if and only if $X\mathrm{adj}(A)X\mathrm{adj}(A)=0$. Since image and kernel of both $X$ and $\mathrm{adj}(A)$ are all one-dimensional, it is immediate to see that this happens if and only if 
\begin{equation}\label{eq:silly}
\mathrm{Im}(\mathrm{adj}(A))=\Ker(X)\qquad\text{ or }\qquad\mathrm{Im}(X)=\Ker(\mathrm{adj}(A))~.
\end{equation}
Now, since $\det(A)=0$ implies $\mathrm{adj}(A)A=A\mathrm{adj}(A)=0$, the relations $\Ker(\mathrm{adj}(A))=\mathrm{Im}(A)$ and $\mathrm{Im}(\mathrm{adj}(A))=\Ker(A)$ hold. Hence $X\in P_{[A]}$ if and only if $\mathrm{Im}(X)=\mathrm{Im}(A)$ or $\Ker(X)=\Ker(A)$, which concludes the proof, by the definition of $\delta$.
\end{proof}

\section{Convexity notions}

In this section we develop the necessary tools to tackle the proof of Thurston's earthquake theorem. 

\subsection{Affine charts}

The starting point of the proof rests in considering the graph of an orientation-preserving homeomorphism  $f:\RP^1\to\RP^1$ as a subset of $\partial_\infty\AdS^3$, and taking its convex hull. However, the convex hull of a set in projective space can be defined in an affine chart, but $\overline\AdS{}^3$ is not contained in any affine chart. The following lemma serves to show that the convex hull of the graph of $f$ is well-defined.

\begin{lemma} \label{lem:disjoint}
Let $f:\mathbb{RP}^1\to \mathbb{RP}^1$ be an orientation-preserving homeomorphism. Then:
\begin{enumerate}
\item There exists a spacelike plane $P_\gamma$ in $\AdS^3$ such that $\partial_\infty P_\gamma\cap\gr(f)=\emptyset$.
\item Moreover, given any point $(x_0,y_0) \notin\gr(f)$, there exists a spacelike plane $P_\gamma$  such that $\partial_\infty P_\gamma\cap\gr(f)=\emptyset$ \emph{and} $(x_0,y_0)\in\partial_\infty P_\gamma$.
\end{enumerate}
\end{lemma}

Before providing the proof, let us discuss an important consequence of the first item. Given a (spacelike) plane $P_\gamma$ in $\AdS^3$, let $\mathcal P_\gamma$ be the unique projective subspace in $\pmat$ that contains $P_\gamma$, which is defined by the equation \eqref{eq:plane} (where $\gamma=[A])$. Let us denote by $\mathcal A_\gamma$ the complement of $\mathcal P_\gamma$, which we will call a \emph{(spacelike) affine chart}. The first item of Lemma \ref{lem:disjoint} can be reformulated as follows:

\begin{corollary} \label{cor:disjoint}
Let $f:\mathbb{RP}^1\to \mathbb{RP}^1$ be an orientation-preserving homeomorphism. There exists a spacelike affine chart $\mathcal A_\gamma$ containing $\gr(f)$.
\end{corollary}

The proof of Lemma \ref{lem:disjoint} below is largely inspired by \cite[Lemma 6.2, Lemma 6.3]{bonschkra}.

\begin{proof}[Proof of Lemma \ref{lem:disjoint}]
Clearly the second item implies the first. However, we will first prove the first item, and then explain how to improve the proof to achieve the second item.

Recall that $\PSL(2,\R)$ acts transitively on pairs of distinct points of $\RP^1\cong\R\cup\{\infty\}$ --- actually, it acts \emph{simply} transitively on \emph{positively oriented triples}. Hence for the first point we may assume, up to the action of the isometry group of $\AdS^3$ by post-composition on $f$ (recall \eqref{eq:action graph}), that $f(0)=0$ and $f(\infty)=\infty$. Then $f$ induces a monotone increasing homeomorphism from $\R$ to $\R$. Since $f(0)=0$, $f$ preserves the two intervals $(-\infty,0)$ and $(0,\infty)$. Let now $\gamma=\mathcal R_i$ be the order--two elliptic isometry fixing $i$. Clearly $\gamma$ is an involution that maps $0$ to $\infty$, and switches the two intervals $(-\infty,0)$ and $(0,\infty)$. Hence $f(x)\neq\gamma(x)$ for all $x\in\R\cup\{\infty\}$, that is, $\gr(f)\cap\gr(\gamma)=\emptyset$. By Lemma \ref{lem:spacelike planes} and the fact that $\gamma$ is an involution,
$\gr(f)\cap \partial_{\infty}P_{\gamma}=\emptyset$.

To prove the second item, we will make full use of the transitivity of the $\PSL(2,\R)$-action on oriented triples, and we will apply both pre and post-composition of an element of $\PSL(2,\R)$. As a preliminary step, let $(x_0,y_0)\notin\gr(f)$, and observe that we can find points $x$ and $x'$ such that $f$ maps the unoriented arc of $\RP^1$ connecting $x$ and $x'$ containing $x_0$ to the unoriented arc connecting $f(x)$ and $f(x')$ \emph{not} containing $y_0$. The proof is just a picture, see Figure \ref{fig:graphs}.  Since $f$ preserves the orientation of $\RP^1$, up to switching $x$ and $x'$, we have that $(x,x_0,x')$ is a positive triple in $\RP^1$, while $(f(x),y_0,f(x'))$ is a negative triple. 

Having made this observation, using simple transitivity on oriented triples we can assume $(x,x_0,x')=(0,1,\infty)$ and $(f(x),y_0,f(x'))=(0,-1,\infty)$. Then the choice $\gamma=\mathcal R_i$ as in the first part of the proof satisfies the condition in the second item as well, since $\gamma(1)=-1$.
\end{proof}

\begin{figure}[htb]
\centering
\includegraphics[scale=.65]{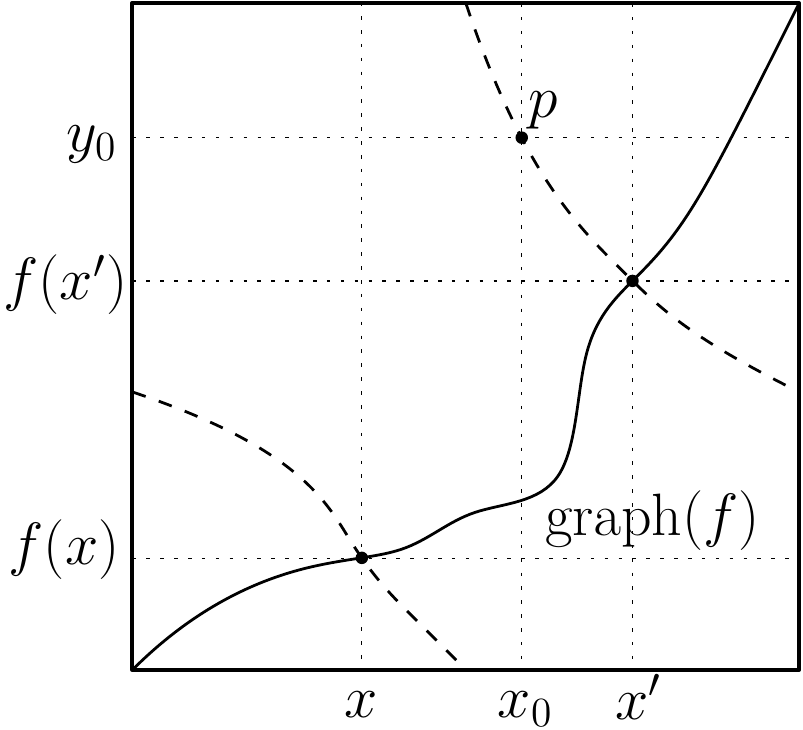}
\caption{The proof of a claim in Lemma \ref{lem:disjoint}, drawn in the torus $\RP^1\times\RP^1$ (identify opposite sides by a translation). Given $p\notin\gr(f)$, consider any orientation-reversing homeomorphism $g$ of $\RP^1$. Then $\gr(f)$ and $\gr(g)$ (dashed) intersect in two points, and let $x,x'$ the corresponding solutions of the equation $f=g$. Then $f$ maps an arc from $x$ to $x'$ containing $x_0$, to an arc from $f\!(x)$ to $f\!(x')$ not containing $y_0$.}
\label{fig:graphs}       
\end{figure}

\subsection{Convex hulls}

Corollary \ref{cor:disjoint} permits to consider the convex hull of $\gr(f)$, in any affine chart $\mathcal A_\gamma$ that contains $\gr(f)$.

\begin{example}\label{ex:plane}
Given $\sigma\in\PSL(2,\R)$, the convex hull of $\gr(\sigma)$ is the closure of the totally geodesic spacelike plane $P_{\sigma^{-1}}$ in $\AdS^3$. Indeed by Lemma \ref{lem:spacelike planes} the boundary at infinity of $P_{\sigma^{-1}}$ equals $\gr(\sigma)$, and moreover $P_{\sigma^{-1}}$  is convex, since spacelike geodesics of $\AdS^3$ (which are the intersections of two  transverse spacelike planes) are lines in an affine chart, and any two points in $\partial_\infty\Hyp^2$ are connected by a geodesic. Hence $P_{\sigma^{-1}}$ is clearly the smallest convex set containing $\gr(\sigma)$. 

This is the only case in which $\gr(f)$ is contained in a plane, and therefore its convex hull has empty interior. If $f$ is not the restriction to $\RP^1$ of an element of $\PSL(2,\R)$, then the convex hull of $\gr(f)$ is a convex body in the affine chart  $\mathcal A_\gamma$.
\end{example}

Let us study one more important property of the convex hull of $\gr(f)$.

\begin{proposition} \label{prop:intersection bdy}
Let $f:\mathbb{RP}^1\to \mathbb{RP}^1$ be an orientation-preserving homeomorphism, let   $P_\gamma$ in $\AdS^3$ be a spacelike plane such that $\partial_\infty P_\gamma\cap\gr(f)=\emptyset$, and let $K$ be the convex hull of $\gr(f)$ in the affine chart $\mathcal A_\gamma$. Then:
\begin{enumerate}
    \item[\textbullet] The interior of $K$ is contained in $\AdS^{3}$.
    \item[\textbullet] The intersection of $K$ with $\partial_\infty\AdS{}^{3}$ equals $\gr(f)$.
\end{enumerate}
In particular, $K\subset\overline\AdS{}^3$.
\end{proposition}

Before proving Proposition \ref{prop:intersection bdy}, we give another technical lemma, which is proved by an argument in a similar spirit as the proof of Lemma \ref{lem:disjoint}.

\begin{lemma}\label{lem:intersection interior}
Let $f:\mathbb{RP}^1\to \mathbb{RP}^1$ be an orientation-preserving homeomorphism and let   $P_\gamma$ in $\AdS^3$ be a spacelike plane such that $\partial_\infty P_\gamma\cap\gr(f)=\emptyset$. Given any two distinct points $(x,f(x))$ and $(x',f(x'))$ in $\gr(f)$, there exists a spacelike plane, disjoint from $P_\gamma$, containing them in its boundary at infinity. 
\end{lemma}

\begin{proof}
Applying the action of $\PSL(2,\R)\times\PSL(2,\R)$ we can assume that $\gamma=\1$. The hypothesis $\partial_\infty P_\1\cap\gr(f)=\emptyset$ then tells us that $f$ has no fixed point. We are looking for a $\sigma\in\PSL(2,\R)$ such that 
\begin{itemize}
\item $P_\1\cap P_{\sigma^{-1}}=\emptyset$;
\item $(x,f(x)),(x',f(x'))\in\partial_\infty P_{\sigma^{-1}}=\gr(\sigma)$.
\end{itemize}
For the first condition to hold, it clearly suffices that the boundaries of $P_\1$ and $P_{\sigma^{-1}}$ do not intersect, that is to say, $\sigma(y)\neq y$ for all $y\in\RP^1$. This is equivalent to saying that $\sigma$ does not have fixed points on $\RP^1$, namely, $\sigma$ is an elliptic isometry. The second condition is equivalent to $\sigma(x)=f(x)$ and $\sigma(x')=f(x')$. 

Since $f$ has no fixed points, $f(x)\neq x$ and $f(x')\neq x'$. There are various cases to distinguish (see also Figure \ref{fig:casesremark}). First, suppose $(x,f(x),x')$ is a positive triple. Then either $(x,f(x'),f(x),x')$ or $(x,f(x),x',f(x'))$ are in cyclic order, because the remaining possibility, namely that $(x,f(x),f(x'),x')$ are in cyclic order, would imply that $f$ has a fixed point. If $(x,f(x'),f(x),x')$ are in cyclic order, then the hyperbolic geodesics $\ell$ connecting $x$ to $f(x)$ and $\ell'$ connecting $x'$ to $f(x')$ intersect, and the order two elliptic isometry $\sigma$ fixing $\ell\cap \ell'$ maps $x$ to $f(x)$ and $x'$ to $f(x')$. If  $(x,f(x),x',f(x'))$ are in cyclic order, then the geodesics $\ell_1$ connecting $x$ to $x'$ and $\ell_2$ connecting $f(x)$ to $f(x')$ intersect, and one can find an elliptic element $\sigma$ fixing $\ell_1\cap\ell_2$ sending $x$ to $f(x)$ and $x'$ to $f(x')$. Second, if $(x,f(x),x')$ is a negative triple, then  the argument is completely analogous. Finally,  there is the possibility that $f(x)=x'$. If $f(x')\neq x$, the $\sigma$ we are looking for is for instance an order--three elliptic isometry with fixed point in the barycenter of the triangle with vertices $x,f(x)=x'$ and $f(x')$. If instead $f(x')=x$, then clearly we can pick any order--two elliptic isometry with fixed point on the geodesic $\ell$ from $x$ to $x'$.  
\end{proof}

In particular, Lemma \ref{lem:intersection interior} shows that given any spacelike affine chart $\mathcal A_\gamma$ containing $\gr(f)$ and any two distinct points in $\gr(f)$, the line connecting them is contained in $\AdS^3\cap \mathcal A_\gamma$ (except for its endpoints, which are in $\partial_\infty\AdS^3$), and is a spacelike geodesic of $\AdS^3$.

\begin{figure}[htb]
\centering
\includegraphics[width=\textwidth]{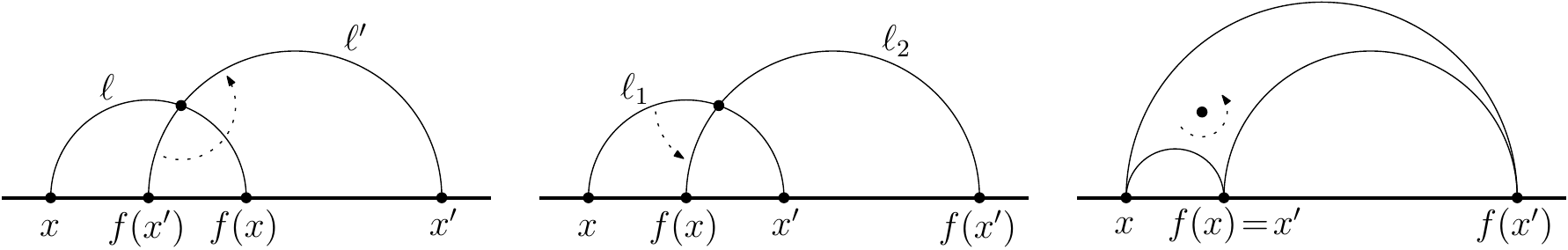}
\caption{Several cases in the proof of Lemma \ref{lem:intersection interior}.}
\label{fig:casesremark}       
\end{figure}

We are now ready to prove Proposition \ref{prop:intersection bdy}.

\begin{proof}[Proof of Proposition \ref{prop:intersection bdy}]
Given a point $p$ in $\partial_{\infty}\AdS^3\setminus \gr(f)$, by the second item of Lemma \ref{lem:disjoint} there exists a spacelike plane $P_{\gamma'}$ passing through $p$
that does not intersect $\gr(f)$. This implies that $P_{\gamma'}\cap K=\emptyset$, and therefore $K\cap \partial_{\infty}\AdS^3=\gr(f)$.
Since $K$ is connected, it is contained in the
closure of one component of the complement of $\partial_\infty\AdS^3$ in $\mathcal A_\gamma$. But $K$ is connected and  intersects $\AdS^3\setminus P_\gamma$ because, by Lemma \ref{lem:intersection interior}, the line segment connecting any two points of $\gr(f)$ in the affine chart $\mathcal A_\gamma$ is contained in $\AdS^3\cap\mathcal A_\gamma$. Hence $K$ is contained in $\overline\AdS{}^{3}$ and its interior is contained in $\AdS^3$.
\end{proof}

By Corollary \ref{cor:disjoint} and Proposition \ref{prop:intersection bdy}, we can now give the following definition:

\begin{definition}
Given an orientation-preserving homeomorphism $f:\mathbb{RP}^1\to \mathbb{RP}^1$, we define $\con(f)$ to be the subset of $\overline\AdS{}^3$ which is obtained as the convex hull of $\gr(f)$ in any spacelike affine chart $\mathcal A_\gamma$ such that $\partial_\infty P_\gamma\cap\gr(f)=\emptyset$.
\end{definition}

The definition is well posed --- that is, it does not depend on the chosen affine chart $\mathcal A_\gamma$ --- because lines and planes are well defined in projective space, hence the change of coordinates from an affine chart to another preserves convex sets. When referring to convexity notions in the following, we will implicitly assume we have chosen a spacelike affine chart $\mathcal A_\gamma$ containing $\gr(f)$.

\subsection{Support planes}

Let us recall a basic notion in convex analysis.
Given a  convex body $K$ in an affine space of dimension three, a \emph{support plane} of $K$ is an affine plane $Q$ such that $K$ is contained in a closed half-space bounded by $Q$, and $\partial K\cap Q\neq \emptyset$. If $p\in \partial  K\cap Q$, one  says that $Q$ is a support plane \emph{at the point} $p$. As a consequence of the Hahn--Banach theorem, there exists a support plane at every point $p\in \partial K$.

We will adopt this terminology for the convex hulls  $\con(f)$ in $\AdS^3$: we say that a totally geodesic plane $P$ is a support plane of  $\con(f)$ (at $p\in\partial\con(f)$) if $p\in\con(f)\cap \overline P\subset\overline\AdS{}^3$ and, in an affine chart containing $\gr(f)$, $\con(f)$ lies  in a closed half-space bounded by the affine plane that contains $P$. As usual, one easily sees that this definition does not depend on the affine chart as long as it contains $\gr(f)$.

\begin{remark}\label{rem:equivalent def support}
Equivalently, we can say that a  totally geodesic plane $P$ is a support plane for $\con(f)$ if there exists a continuous family $\{P_t\}_{t\in [0,\epsilon)}$ of  totally geodesic planes, pairwise disjoint in $\overline\AdS{}^3$, such that $P_0=P$ and  $P_{t}\cap \con(f)=\emptyset$ for $t> 0$.
\end{remark}

Also, recall that we have the following identity for convex hulls: if $X$ is a set, $\con(X)$ its convex hull and $Q$ an affine support plane for $\con(X)$, then $Q\cap \con(X)=\con(Q\cap X)$. Applying this identity in our setting, we obtain for any totally geodesic support plane $P$:
\begin{equation}\label{eq:shape}
P\cap\con(f)=\con(\partial_\infty P\cap\gr(f))~.
\end{equation}

In the following proposition, we see that all support planes of $\con(f)$ are allowed to be spacelike, and  lightlike only if they touch $\con(f)$ at a boundary point.

\begin{proposition}\label{prop:support planes space or light}
Let $f:\mathbb{RP}^1\to \mathbb{RP}^1$ be an orientation-preserving homeomorphism, and let $P$ be a support plane of $\con(f)$ at a point $p\in\partial\con(f)$. Then:
\begin{itemize}
\item If $p\in \AdS^3$, then $P$ is a spacelike plane.
\item If $p\in \partial_\infty\AdS^3$, then $P$ is either spacelike or lightlike.
\end{itemize}
\end{proposition}
\begin{proof}
The basic observation is that if $P$ is a support plane, then $\partial_\infty P$ and $\gr(f)=\con(f)\cap\partial_\infty\AdS^3$ do not intersect transversely. To clarify this notion, we say that an intersection point $p\in \partial_\infty P\cap\gr(f)$ is transverse if, for a small neighbourhood $U$ of $p$ such that $(\gr(f)\setminus p)\cap U$ has two connected components, these two connected components are contained in different connected components of $U\setminus \partial_\infty P$. From Lemma \ref{lem:timelike planes}, if $P$ is timelike, then $\partial_\infty P$ is the graph of an orientation-reversing homeomorphism of $\RP^1$, hence it intersects $\gr(f)$ transversely. From Lemma \ref{lem:lightlike planes}, if $P$ is lightlike, then $\partial_\infty P$ is the union of the two circles $\{x\}\times\RP^1$ and $\RP^1\times\{y\}$. So if $p\in \partial_\infty P\cap\gr(f)$ and $p$ is not the point $p_0=(x,y)$, then $\partial_\infty P$ and $\gr(f)$ intersect transversely. So the sole possibility for $P$ to be a lightlike support plane is to intersect $\gr(f)$ only at the point $p_0$. It remains to show that $P\cap\con(f)$ consists only of the point $p_0$, that is, it does not contain any point of $\AdS^3$. By contradiction, if $q\in  P\cap\con(f)$ is different from $p_0$, then by \eqref{eq:shape} $\partial_\infty P\cap\gr(f)$ would contain another point different by $p_0$ as well, because the left-hand side must contain not only $p_0$ but also $q$. This would give a contradiction as above.
\end{proof}

Given a spacelike support plane $P$ of $\con(f)$ at a point $p$, we say that $P$ is a \emph{future} (resp. \emph{past}) support plane if in a small simply connected neighbourhood $U$ of $p$ in $\overline\AdS{}^3$, $\con(f)$ is contained in the closure of the connected component of $U\setminus P$ which is in the past (resp. future) of $P$. This means that there exist future-oriented (resp. past-oriented) timelike curves in $U$ leaving $\con(f)\cap U$ and reaching $P\cap U$. 

Clearly $\con(f)$ cannot have a future and past support plane at $p$ at the same time, unless $\con(f)$ has empty interior, which is precisely the situation when $f$ is an element of $\PSL(2,\R)$ as in Example \ref{ex:plane}. In the following we will always assume $\mathrm{int}\,\con(f)\neq\emptyset$.
As a consequence of the previous discussion, we have the following useful statement on the convergence of support planes.

\begin{lemma}\label{lem: convergence support planes}
Let $f:\mathbb{RP}^1\to \mathbb{RP}^1$ be an orientation-preserving homeomorphism which is not in $\PSL(2,\R)$, $p_n$  a sequence of points in $\partial\con(f)$, and $P_{\gamma_n}$  a  sequence of future (resp. past) spacelike support planes at $p_n$, for $\gamma_n\in\PSL(2,\R)$. Up to extracting a subsequence, we can assume $p_n\to p$ and $P_{\gamma_n}\to P$. Then:
\begin{itemize}
\item If $p\in\AdS^3$, then $P=P_{\gamma}$ is a future (resp. past) support plane of $\con(f)$, for $\gamma_n\to \gamma\in\PSL(2,\R)$.
\item If $p\in\partial_\infty\AdS^3$, then either $P$ is a lightlike plane whose boundary is the union of two circles meeting at $p$, or the conclusion of the previous point holds.
\end{itemize}
\end{lemma}
\begin{proof}
The proof is straightforward, having developed all the necessary elements above. It is clear that we can extract converging subsequences from $p_n$ and $P_{\gamma_n}$, by compactness of $\con(f)$ and of the space of planes in projective space. Also, the limit of the sequence of support planes $P_{\gamma_n}$ at $p_n$ is a support plane $P$ at $p$, since both conditions that $p_n\in\con(f)$ and that $\con(f)$ is contained in a closed half-space bounded by $P_{\gamma_n}$ are closed conditions. By Proposition \ref{prop:support planes space or light}, 
 if the limit $p$ is in $\AdS^3$, then $P$ is a spacelike support plane, which is of course future (resp. past) if all the $P_{\gamma_n}$ are future (resp. past). This situation can also occur analogously if $p\in\partial_\infty\AdS^3$; the other possibility being that $P$ is lightlike, and in this case the proof of Proposition \ref{prop:support planes space or light} shows that $P=P_{[A]}$ if $p$ is represented by the projective class of the rank--one matrix $A$.
\end{proof}

\begin{corollary}\label{cor:past or future}
Let $f:\mathbb{RP}^1\to \mathbb{RP}^1$ be an orientation-preserving homeomorphism which is not in $\PSL(2,\R)$. Then $\partial\con(f)$ is the disjoint union of $\gr(f)=\con(f)\cap\partial_\infty\AdS^3$ and of two topological discs,  of which one only admits future support plane, and the other only admits past support planes.
\end{corollary}
\begin{proof}
It is a basic fact in convex analysis that $\partial\con(f)$ is homeomorphic to $\mathbb S^2$; by Proposition \ref{prop:intersection bdy}, its intersection with $\partial_\infty\AdS^3$ equals $\gr(f)$ and is therefore a simple closed curve. By the Jordan curve theorem, the complement of $\gr(f)$ is the disjoint union of two topological discs, each of which is contained in $\AdS^3$ again by Proposition \ref{prop:intersection bdy}. 

By Lemma \ref{lem: convergence support planes}, the set of points $p\in\partial\con(f)$ admitting a future support plane is closed. But it is also open because its complement is the set of points admitting a past support plane, for which the same argument applies. Hence each connected component of the complement of $\gr(f)$ admits only   future support planes, or only past support planes. Finally, $\con(f)$ necessarily admits both a past and a future support plane, otherwise it would not be compact in an affine chart. This concludes the proof.
\end{proof}

By virtue of Corollary \ref{cor:past or future}, we will call the connected component of $\partial\con(f)\setminus\gr(f)$ that only admits future support planes the \emph{future boundary component}, and denote it by $\partial_+\con(f)$; similarly, the connected component that only admits past support planes is the \emph{past boundary component}, denoted by $\partial_-\con(f)$.

\subsection{Left and right projections}

We are now ready to introduce the \emph{left and right projections}, which will play a central role in the proof of the earthquake theorem. These are maps 
$$\pi_l^\pm:\partial_\pm\con(f)\to\Hyp^2\qquad \pi_r^\pm:\partial_\pm\con(f)\to\Hyp^2$$
defined on the future or past components of $\partial\con(f)$, constructed as follows. Given a point $p\in\partial_\pm\con(f)$, let $P$ be a support plane of $\con(f)$ at $p$. By Proposition \ref{prop:support planes space or light}, the support plane is necessarily spacelike, hence of the form $P=P_\gamma$ for some $\gamma\in\PSL(2,\R)$. 

\begin{remark}\label{rmk:choice}
It is important to remark here that $P_\gamma$ might not be unique, if $\partial_\pm\con(f)$ is not $C^1$ at $p$. Hence we \emph{choose}  a support plane $P_\gamma$ at $p$. Moreover we require that the choice of support planes is made so that the support plane is constant on any connected component of the subset of $\partial_\pm\con(f)$ consisting of those points that admit more than one support plane. The definition of the projections then \emph{depends} (although quite mildly, see Corollary \ref{cor: image does not depend} below) on the choice of $P_\gamma$.
\end{remark}

Now, having chosen the support plane $P_\gamma$ at $p$, left or right multiplication by $\gamma^{-1}$ maps $\gamma$ to $\1$, and therefore maps $P_\gamma$ to $P_\1$, which we recall from Example \ref{ex:traceless} is the space of order--two elliptic elements and is therefore naturally identified with $\Hyp^2$ via the map $\Fix:P_\1\to\Hyp^2$. 

Denote by $L_{\gamma^{-1}}:\PSL(2,\R)\to\PSL(2,\R)$ and $R_{\gamma^{-1}}:\PSL(2,\R)\to\PSL(2,\R)$ the left and right multiplications by $\gamma^{-1}$; in other words, there are the actions of the elements $(\gamma,\1)$ and $(\1,\gamma^{-1})$ of $\PSL(2,\R)\times\PSL(2,\R)$. By what we said above, $L_{\gamma^{-1}}(p)$ and $R_{\gamma^{-1}}(p)$ are elements of $P_\1$, since $p\in P_\gamma$, and $L_{\gamma^{-1}}(p)$ (resp. $R_{\gamma^{-1}}(p)$) maps bijectively $P_\gamma$ to $P_\1$. We can finally
define:
\begin{equation}\label{defi projections}
\pi_l^\pm(p)=\Fix(R_{\gamma^{-1}}(p))\qquad \pi_r^\pm(p)=\Fix(L_{\gamma^{-1}}(p))~.
\end{equation}

It might seem counterintuitive to define the left projection using right multiplication, and vice versa. However, this is the most natural choice by virtue of the property of Lemma \ref{lem:left right proj} below. Another reason to justify this choice is that these projections can be naturally seen as the left and right components of the \emph{Gauss map} of spacelike surfaces in $\AdS^3$ with values in the space of timelike geodesics of $\AdS^3$, which is naturally identified with $\Hyp^2\times\Hyp^2$, see \cite[Section 6.3]{survey} for more details and for several other equivalent definitions.

\begin{lemma}\label{lem:left right proj}
Let $f:\mathbb{RP}^1\to \mathbb{RP}^1$ be an orientation-preserving homeomorphism, and let $(\alpha,\beta)\in\PSL(2,\R)\times\PSL(2,\R)$. Let us denote $K=\con(f)$ and $\hat K=(\alpha,\beta)\cdot\con(f)$ and let $\pi_l^\pm,\pi_r^\pm:\partial_\pm K\to\Hyp^2$ and $\hat\pi_l^\pm,\hat\pi_r^\pm:\partial_\pm \hat K\to\Hyp^2$ be the left and right projections of $K$ and $\hat K$ respectively. Then
\begin{equation}\label{eq:lem:left right proj}
\hat\pi_l^\pm\circ (\alpha,\beta)=\alpha\circ\pi_l^\pm \qquad \hat\pi_r^\pm\circ (\alpha,\beta)=\beta\circ \pi_r^\pm~.
\end{equation}
\end{lemma}

To clarify the statement, let us remark that the isometry $(\alpha,\beta)$ maps a point $p\in K$ to a point $\hat p\in \hat K$, and maps support planes at $p\in K$ to support planes at $\hat p$. Hence the relation \eqref{eq:lem:left right proj} holds when we consider the projections $\hat \pi_l^\pm$ and  $\hat \pi_r^\pm$ defined with the choice of support planes of $\hat K$ given by the images $\hat P$ of the  support planes $P$ chosen in the definitions of $\pi_l^\pm$ and $ \pi_r^\pm$.

\begin{proof}
As remarked above, for any $p\in \partial^\pm K$, we have  $\hat p:=(\alpha,\beta)\cdot p\in \hat K$, and for a chosen support plane $P=P_{\gamma}$ for $K$ at $p$, $(\alpha,\beta)\cdot P=P_{\hat \gamma}$ is the chosen support plane  for $\hat K$ at $\hat p$. By the duality, $\hat\gamma=(\alpha,\beta)\cdot\gamma=\alpha\gamma\beta^{-1}$. Hence we have:
\begin{align*}
\hat \pi_l^\pm(\hat p)&=\Fix\left(R_{\hat\gamma^{-1}}(\hat p)\right)=\Fix\left(R_{(\beta\gamma^{-1}\alpha^{-1})}(\alpha p\beta^{-1})\right) \\
&=\Fix\left(R_{(\gamma^{-1}\alpha^{-1})}(\alpha p)\right)= \Fix\left(\alpha\circ R_{\gamma^{-1}}( p)\circ\alpha^{-1}\right)\\
&=\alpha(\Fix(R_{\gamma^{-1}}( p)))=\alpha \pi_l^\pm(p)~.
\end{align*}
The computation is completely analogous for the right projection.
\end{proof}

\begin{example}\label{ex: extension isometry}
The simplest example that we can consider is the situation where $f=\sigma\in\PSL(2,\R)$, so that $\con(f)=P_{\sigma^{-1}}$ as in Example \ref{ex:plane}. This case is somehow degenerate, because $\con(\sigma)$ has empty interior, hence Corollary \ref{cor:past or future} does not hold and it does not quite make sense to talk about the future and past components of the boundary. However, we can still define the left and right projections. Since $P_{\sigma^{-1}}$ itself is the unique support plane at any of its points, from \eqref{defi projections} we have the following simple expressions for the left and right projections $\pi_l,\pi_r:P_{\sigma^{-1}}\to\Hyp^2$.
\begin{equation}\label{eq:ext proj plane}
\pi_l(p)=\Fix(p\circ\sigma)\qquad \pi_r(p)=\Fix(\sigma\circ p)~.
\end{equation}
Observe that $\pi_l$ and $\pi_r$ extend to the boundary of $P_{\sigma^{-1}}$: recalling that the boundary of $P_{\sigma^{-1}}$ is the graph of $\sigma$ (Lemma \ref{lem:spacelike planes}), we have  
\begin{equation}\label{eq:ext bdy plane}
\pi_l(x,\sigma(x))=x\qquad \pi_r(x,\sigma(x))=\sigma(x)~.
\end{equation}
Equation \eqref{eq:ext bdy plane} is indeed immediately checked when $\sigma=\1$, because in that case $\pi_l$ and $\pi_r$ coincide with  the fixed point map $\Fix:P_\1\to\Hyp^2$, and we have already observed in Example \ref{ex:plane}, using \eqref{eq:convergence}, that $\Fix$ extends to the map $(x,x)\mapsto x$ from $\partial_\infty P_\1$ to $\partial_\infty\Hyp^2$. The general case of Equation \eqref{eq:ext bdy plane} then follows from Equations \eqref{eq:action graph} and  \eqref{eq:lem:left right proj}, that is, by observing that the isometry $(\1,\sigma)$ maps $\gr(\1)$ to $\gr(\sigma)$ and $P_\1$ to $P_{\sigma^{-1}}$. 

Finally, we can compute the map of $\Hyp^2$ obtained by composing the inverse of the left projection with the right projection. Indeed, this  is induced by the map $P_\1\to P_\1$ sending an order--two elliptic element $\mathcal R=p\circ \sigma \in P_\1$ to $\sigma\circ p=\sigma\circ\mathcal R\circ \sigma^{-1}$. Hence we have
\begin{equation}\label{eq:compose plane}
\pi_r\circ\pi_l^{-1}=\sigma:\Hyp^2\to\Hyp^2~.
\end{equation}
\end{example}

In conclusion, the composition $\pi_r\circ\pi_l^{-1}$ is an isometry and its extension to $\partial_\infty\Hyp^2$ is precisely the map $f=\sigma$ of which $\partial_\infty P_{\sigma^{-1}}$ is the graph. In the next sections we will see that this fact is extremely general, that is, for any orientation-preserving homeomorphism of the circle $f$,  the compositions $\pi_r^\pm\circ(\pi_l^\pm)^{-1}$ associated with $\partial_\pm\con(f)$ will be the left and right earthquake maps extending $f$.

\section{The case of two spacelike planes}\label{sec:example}

Before moving to the proof of Thurston's earthquake theorem, we will now consider another very concrete example, which is only slightly more complicated than Example \ref{ex: extension isometry}. Nevertheless, we will see that this example represents a very general situation, and its comprehension is the essential step towards the proof of the full theorem.

\subsection{The fundamental example}

The idea here is to consider piecewise totally geodesic surfaces in $\AdS^3$, which are obtained as the union of two connected subsets, each contained in a totally geodesic spacelike plane, meeting along a common geodesic. See Figure \ref{fig:pleated}.

To formalize this idea, we will consider the union of two half-planes, each contained in a totally geodesic spacelike plane $P_{\gamma_1}$ or $P_{\gamma_2}$. The first important observation is the following.

\begin{lemma}\label{lemma intersection two planes}
Let $\gamma_1\neq \gamma_2\in\PSL(2,\R)$. Then $P_{\gamma_1}$ and $P_{\gamma_2}$ intersect in $\AdS^3$ if and only if $\gamma_2\circ \gamma_1^{-1}$ is a hyperbolic isometry.
\end{lemma}
\begin{proof}
Since $P_{\gamma_i}$ is the convex envelope of $\partial_\infty P_{\gamma_i}=\gr(\gamma_i^{-1})$ (Example \ref{ex:plane}), the closures $\overline P_{\gamma_1}$ and $\overline P_{\gamma_2}$ intersect in $\overline\AdS{}^3$ if and only if $\gr(\gamma_1^{-1})\cap\gr(\gamma_2^{-1})\neq\emptyset$. Moreover, by \eqref{eq:shape}, $P_{\gamma_1}$ and $P_{\gamma_2}$ intersect in $\AdS^3$ if and only $\gr(\gamma_1^{-1})\cap\gr(\gamma_2^{-1})$ contains at least two different points. 

Now, $(x,y)\in\RP^1\times\RP^1$ is in $\gr(\gamma_1^{-1})\cap\gr(\gamma_2^{-1})$ if and only if $y=\gamma_1^{-1}(x)=\gamma_2^{-1}(x)$, which is equivalent to the condition that $x$ is a fixed point of $\gamma_2\circ\gamma_1^{-1}$. But $\gamma_2\circ\gamma_1^{-1}$ is an element of $\PSL(2,\R)$, hence it has two fixed points in $\RP^1$ if and only if it is a hyperbolic isometry.
\end{proof}

Now, let $I_1$ and $I_2$ be two closed intervals in $\RP^1$ such that $\RP^1=I_1\cup I_2
$ and $I_1\cap I_2$ consists precisely of the two fixed points of $\gamma_2\circ \gamma_1^{-1}$. Clearly there are two possibilities to produce a homeomorphism of $\RP^1$ by combining the restrictions of  $\gamma_1^{-1}$ and $\gamma_2^{-1}$ to the intervals $I_j$'s, that is:
\begin{equation}\label{eq:f+-}
f^+_{\gamma_1,\gamma_2}(x)=\begin{cases} \gamma_1^{-1} &\text{ if }x\in I_1 \\ \gamma_2^{-1} &\text{ if }x\in I_2 \end{cases}
\qquad\text{and}\qquad f^-_{\gamma_1,\gamma_2}(x)=\begin{cases}  \gamma_2^{-1} &\text{ if }x\in I_1 \\ \gamma_1^{-1} &\text{ if }x\in I_2 \end{cases}~.
\end{equation}
One easily checks that $f^\pm_{\gamma_1,\gamma_2}$ actually are orientation-preserving homeomorphisms, since $\gamma_1^{-1}$ and $\gamma_2^{-1}$ map homeomorphically the intervals $I_1$ and $I_2$ to the same intervals $J_1:=\gamma_1^{-1}(I_1)=\gamma_2^{-1}(I_1)$ and $J_2:=\gamma_1^{-1}(I_2)=\gamma_2^{-1}(I_2)$, which intersect only at their endpoints.

Let us also denote by $D_i$ the convex hull of $I_i$ in $\Hyp^2$, and by $\ell=D_1\cap D_2$ the axis of $\gamma_2\circ \gamma_1^{-1}$.

\begin{proposition}\label{prop:past vs future}
Suppose that $\gamma_2\circ \gamma_1^{-1}$ is a hyperbolic isometry that translates along $\ell$ to the left, as seen from $D_1$ to $D_2$. Then:
\begin{itemize}
\item The future boundary component $\partial_+\con(f^+_{\gamma_1,\gamma_2})$ coincides with the union of the convex envelope of $\gr(\gamma_1^{-1}|_{I_1})$ and of the convex envelope of $\gr(\gamma_2^{-1}|_{I_2})$.
\item The past boundary component $\partial_-\con(f^-_{\gamma_1,\gamma_2})$ is the union of the convex envelope of $\gr(\gamma_1^{-1}|_{I_2})$ and of the convex envelope of $\gr(\gamma_2^{-1}|_{I_1})$. 
\end{itemize}
If instead $\gamma_2\circ \gamma_1^{-1}$ translates along $\ell$ to the right as seen from $D_1$ to $D_2$, then:
\begin{itemize}
\item The past boundary component $\partial_-\con(f^+_{\gamma_1,\gamma_2})$ coincides with the union of the convex envelope of $\gr(\gamma_1^{-1}|_{I_1})$ and of the convex envelope of $\gr(\gamma_2^{-1}|_{I_2})$.
\item The future boundary component $\partial_+\con(f^-_{\gamma_1,\gamma_2})$ is the union of the convex envelope of $\gr(\gamma_1^{-1}|_{I_2})$ and of the convex envelope of $\gr(\gamma_2^{-1}|_{I_1})$. 
\end{itemize}
\end{proposition}

\begin{figure}[htb]
\centering
\begin{minipage}[c]{.5\textwidth}
\centering
\includegraphics[height=7cm]{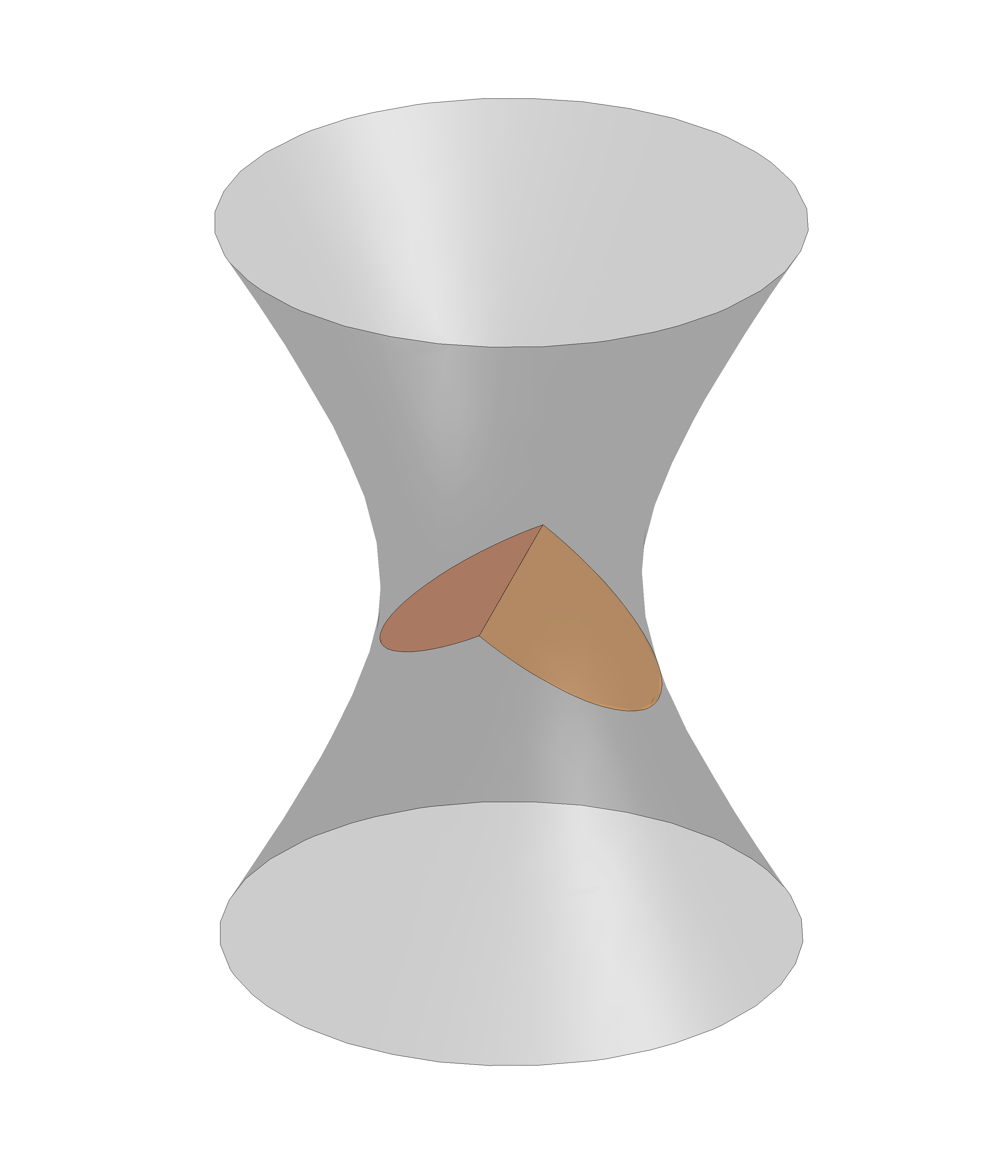}
\end{minipage}%
\begin{minipage}[c]{.5\textwidth}
\centering
\includegraphics[height=7cm]{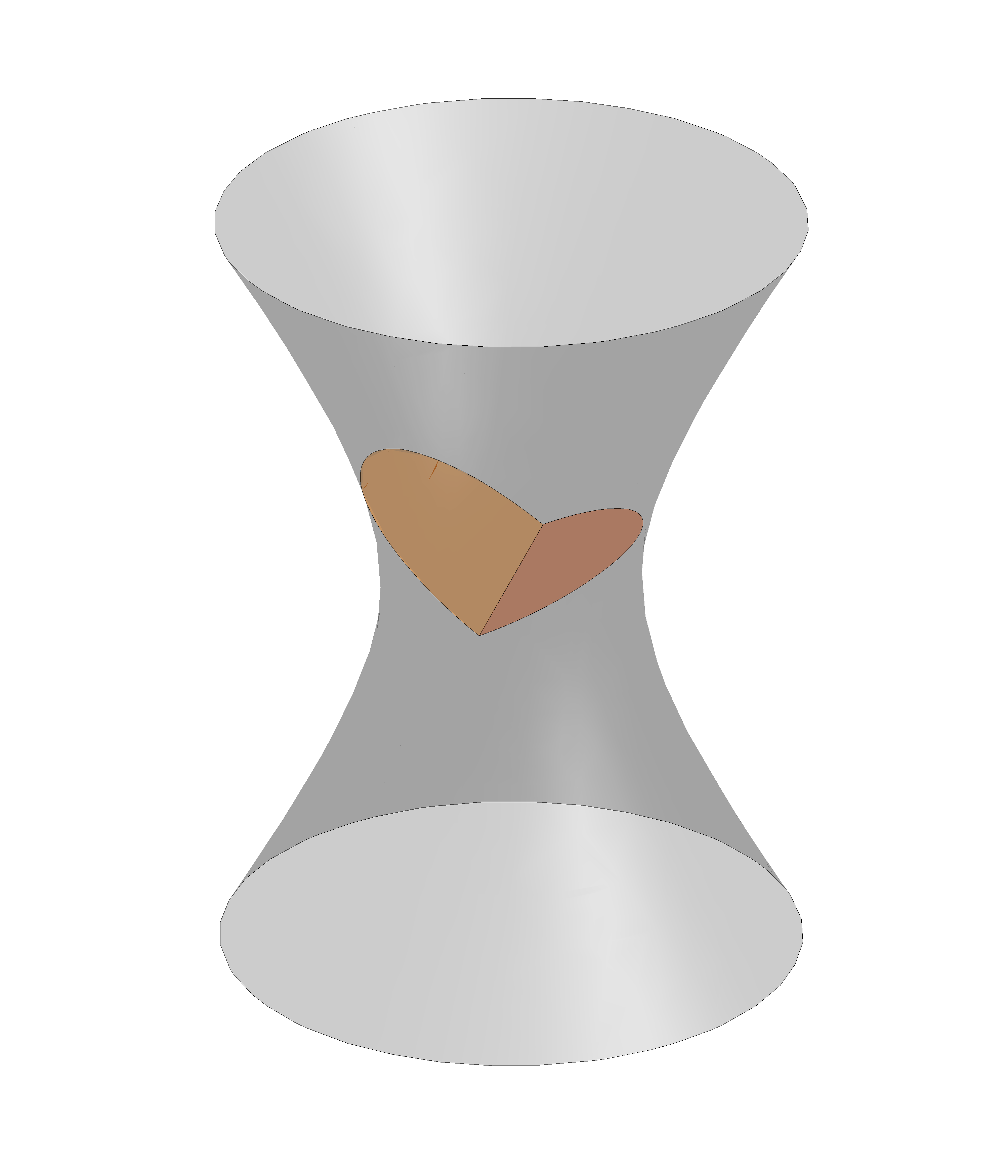}
\end{minipage}
\caption{Given two elements $\gamma_1,\gamma_2$ such that $\gamma_2\circ \gamma_1^{-1}$ is a hyperbolic isometry, there are two possible configurations. On the left, we see the future boundary component of $\con(f^\pm_{\gamma_1,\gamma_2})$, on the right the past boundary component of $\con(f^\mp_{\gamma_1,\gamma_2})$.}
\label{fig:pleated}       
\end{figure}

\begin{proof}
Let us consider the case where $\gamma_2\circ \gamma_1^{-1}$ translates  to the left along $\ell$, and let us prove the first item. Let $x,x'$ be the fixed points of $\gamma_2\circ \gamma_1^{-1}$, and let $y=\gamma_1^{-1}(x)=\gamma_2^{-1}(x)$ and $y'=\gamma_1^{-1}(x')=\gamma_2^{-1}(x')$. Then the convex envelope of $\gr(\gamma_i^{-1}|_{I_i})$ is a half-plane $A_i$ in $P_{\gamma_i}$ bounded by the geodesic $P_{\gamma_1}\cap P_{\gamma_2}$, which has endpoints $(x,y)$ and $(x',y')$.  Clearly both the convex envelope of $\gr(\gamma_1^{-1}|_{I_1})$ and the convex envelope of $\gr(\gamma_2^{-1}|_{I_2})$ are  contained in $\con(f^+_{\gamma_1,\gamma_2})$. 

Nevertheless, we can be more precise.  We claim that $P_{\gamma_1} $ and $P_{\gamma_2}$ are \emph{future support planes} for $\con(f^+_{\gamma_1,\gamma_2})$. This claim implies that the union of $A_1$ and $A_2$ is contained in the future boundary component $\partial_+\con(f^+_{\gamma_1,\gamma_2})$, because every point $p\in A_1\cup A_2$ admits a future support plane through $p$, which is either $P_{\gamma_1} $ or $P_{\gamma_2}$. However $A_1\cup A_2$ is a topological disc in $\partial_+\con(f^+_{\gamma_1,\gamma_2})$, whose boundary is precisely the curve $\gr(f^+_{\gamma_1,\gamma_2})$ by construction. Hence the claim will imply that $A_1\cup A_2=\partial_+\con(f^+_{\gamma_1,\gamma_2})$.

We prove the claim for $P_{\gamma_1}$, the proof for $P_{\gamma_2}$ being completely analogous.
 it is convenient to assume that $\gamma_1=\1$ and $\gamma_2=\gamma$ is a hyperbolic isometry with fixed points $x$ and $x'$, translating to the left seen from $D_1$ to $D_2$. Indeed, we can apply the isometry $(\1,\gamma_1)$, which sends $P_{\gamma_1}$ to $P_\1$, $P_{\gamma_2}$ to $P_{\gamma_2\gamma_1^{-1}}$, and (by \eqref{eq:action graph}) $\gr(f^+_{\gamma_1,\gamma_2})$ to $\gr(f^+_{\1,\gamma_2\gamma_1^{-1}})$.

Having made this assumption, consider a path $\sigma_t$, for $t\in [0,\epsilon)$ of elliptic elements fixing a given point $z_0\in\Hyp^2$, that rotate clockwise by an angle $t$. As in the proof of Lemma \ref{lemma intersection two planes}, the planes $P_{\sigma_t}$ are pairwise disjoint in $\overline\AdS{}^3$, because  $\sigma_{t_2}\circ\sigma_{t_1}^{-1}$ is still an elliptic element fixing $z_0$ for $t_1\neq t_2$, hence it has no fixed point in $\RP^1$. 
Moreover, observe that $\gamma^{-1}$ is an isometry fixing $\ell$ and translates along $\ell$ on the right as seen from $D_1$ to $D_2$. 
Since $f^+_{\1,\gamma}$ equals the identity on $I_1$ and  $\gamma^{-1}$ on $I_2$, it fixes $I_1$ pointwise, and moves points of $I_2$ clockwise. In particular, the equation $f^+_{\1,\gamma}(x)=\sigma_t^{-1}(x)$ has no solutions for $t>0$, because $\sigma_t^{-1}=\sigma_{-t}$ moves all points counterclockwise if $t$ is positive. 
This shows that  $P_{\sigma_t}\cap \con(f^+_{\1,\gamma})=\emptyset$ for $t>0$, and thus $P_\1$ is a support plane for $\con(f^+_{\1,\gamma})$ by Remark \ref{rem:equivalent def support}. Moreover it is a future support plane: indeed one can check (for instance using \eqref{eq:product}) that $\sigma_{t+\pi/2}=\mathcal R_{z_0}\circ\sigma_t\in P_{\sigma_t}$, and the path $t\mapsto\sigma_t$ is future-directed  because, from the discussion after \eqref{eq:orientation}, its tangent vector is future-directed, hence $\con(f^+_{\1,\gamma})$ is locally  in the past of $P_\1$. 

This concludes the proof of the first point. The other cases are completely analogous.

\end{proof}

See also Figure \ref{fig:pleated} to visualize the different configurations. The following is an important consequence of the proof of Proposition \ref{prop:past vs future}. 

\begin{corollary}\label{cor:all support}
Suppose that $\gamma_2\circ \gamma_1^{-1}$ is a hyperbolic isometry that translates along $\ell$ to the left (resp. right), as seen from $D_1$ to $D_2$, and write $\gamma_2\circ \gamma_1^{-1}=\exp(\mathfrak a)$ for $\mathfrak a\in\psl(2,\R)$. Let $p$ be a point in the future (resp. past) boundary component of $\con(f^+_{\gamma_1,\gamma_2})$. Then:
\begin{itemize}
\item If $p\in \mathrm{int}(A_1)$, then $P_{\gamma_1}$ is the unique support plane of $\con(f^+_{\gamma_1,\gamma_2})$ at $p$.
\item If $p\in \mathrm{int}(A_2)$, then $P_{\gamma_2}$ is the unique  support plane of $\con(f^+_{\gamma_1,\gamma_2})$ at $p$.
\item If $p\in A_1\cap A_2=P_{\gamma_1}\cap P_{\gamma_2}$, then the support planes of $\con(f^+_{\gamma_1,\gamma_2})$ at $p$ are precisely those of the form $P_{\sigma\gamma_1}$ where $\sigma=\exp(t\mathfrak a)$ for $t\in [0,1]$.
\end{itemize}
\end{corollary}
Recall the notation from the proof of Proposition \ref{prop:past vs future}: $A_i\subset P_{\gamma_i}$ is the convex envelope of $\gr(\gamma_i^{-1}|_{I_i})$, which is a half-plane bounded by the geodesic $P_{\gamma_1}\cap P_{\gamma_2}$.

Of course we could provide an analogous statement for $\con(f^-_{\gamma_1,\gamma_2})$, but we restrict to $f^+_{\gamma_1,\gamma_2}$ for simplicity.

\begin{proof}
From Proposition \ref{prop:past vs future},  the pleated surface which is obtained as the  union of  $A_1\subset P_{\gamma_1}$ and $A_2\subset P_{\gamma_2}$ coincides with $\partial_+\con(f^+_{\gamma_1,\gamma_2})$ if $\gamma_2\circ \gamma_1^{-1}$ is a hyperbolic isometry that translates along $\ell$ to the left, and with $\partial_-\con(f^+_{\gamma_1,\gamma_2})$ if it translates to the right, by Proposition \ref{prop:past vs future}.

The first two items are then obvious, since $P_{\gamma_i}$ are smooth surfaces, hence $A_i$ is smooth at any interior point, and therefore has a unique support plane there. The last item can be proved in the same spirit as Proposition \ref{prop:past vs future}. First, we can assume $\gamma_1=\1$ and $\gamma_2=\gamma$ is a hyperbolic isometry translating on the left (resp. right) along $\ell$. By \eqref{eq:shape}, if $P_\sigma$ is a support plane at $p$, then $p$ is in the convex hull of the pairs $(y,\sigma^{-1}(y))$ where $y$ satisfies $\sigma^{-1}(y)=f^\pm_{\1,\gamma}(y)$. The only possibility is then that $p$ lies in the geodesic connecting the points $(x,x)$ and $(x',x')$ in $\RP^1\times\RP^1$, where $x$ and $x'$ are the fixed points of $\gamma$. Hence $\sigma$ must have the same fixed points of $\gamma$. That is, $\sigma$ is a hyperbolic isometry with axis $\ell$ (or the identity). Moreover, by an analogous argument as in Proposition \ref{prop:past vs future}, $P_\sigma$ is in the future (resp. past) of $\con(f^+_{\gamma_1,\gamma_2})$ if and only if $\sigma$ translates on the left (resp. right), and its translation length is less than that of $\gamma$. Hence $\sigma$ is of the form $\exp(t\mathfrak a)$ for $t\in [0,1]$.
\end{proof}

\subsection{Simple earthquake}

We can now conclude the study of orientation-preserving homeomorphisms obtained by combining two elements of $\PSL(2,\R)$. The following proposition shows that in that situation, the composition of the projections $\pi_l^\pm$ and $\pi_r^\pm$ provide the earthquake map as in Example \ref{ex:simple}. This is not interesting in its own, since we recover a simple earthquake map which we had already defined explicitly. However, the following proposition will be an important tool to complete the proof of the earthquake theorem in Section \ref{sec:proof}.

\begin{proposition}\label{prop:proj simple}
Let $\gamma_1,\gamma_2\in\PSL(2,\R)$ be such that $\gamma_2\circ\gamma_1^{-1}$ is a hyperbolic isometry, and let $\pi_l^\pm,\pi_r^\pm$ be the projections associated with the convex envelope of $f^+_{\gamma_1,\gamma_2}$. Then:
\begin{enumerate}
\item $\pi_l^\pm,\pi_r^\pm:\partial_\pm\con(f^+_{\gamma_1,\gamma_2})\to\Hyp^2$ are bijections;
\item Assume $\gamma_2\circ\gamma_1^{-1}$ translates along $\ell$ to the right (resp. left), as seen from $D_1$ to $D_2$. Then the composition $\pi_r^-\circ (\pi^-_l)^{-1}:\Hyp^2\to\Hyp^2$ (resp. $\pi_r^+\circ (\pi^+_l)^{-1}:\Hyp^2\to\Hyp^2$)  is a left (resp. right) earthquake map extending $f^+_{\gamma_1,\gamma_2}$.
\end{enumerate}
\end{proposition}
Again, we considered the case of $f^+_{\gamma_1,\gamma_2}$ for the sake of simplicity, but one could give an analogous statement for $f^-_{\gamma_1,\gamma_2}$. Moreover, we remark that Proposition \ref{prop:proj simple} holds for \emph{any choice} of support planes that is needed to define the projections.
\begin{proof}
For the first point, recall that $A_i\subset P_{\gamma_i}$, and that the union $A_1\cup A_2$ is the past (resp. future) boundary component of $\con(f^+_{\gamma_1,\gamma_2})$ if $\gamma_2\circ\gamma_1^{-1}$ translates along $\ell$ to the right (resp. left).

Hence $(\pi_l^\pm)_{\mathrm{int}(A_i)}$ and $(\pi_r^\pm)|_{\mathrm{int}(A_i)}$ are the restrictions of the projections associated with the totally geodesic plane $P_{\gamma_i}$, which are described in Example \ref{ex: extension isometry}. In particular, $(\pi_l^\pm)_{\mathrm{int}(A_i)}$ and $(\pi_r^\pm)|_{\mathrm{int}(A_i)}$ are the restrictions to $\mathrm{int}(A_i)$ of global isometries of $\AdS^3$ (defined by multiplication on the left or on the right by $\gamma_i^{-1}$) sending $P_{\gamma_i}$ to $P_\1$, post-composed with the usual isometry $\Fix:P_\1\to\Hyp^2$. As a consequence, $(\pi_l^\pm)_{\mathrm{int}(A_i)}$ and $(\pi_r^\pm)|_{\mathrm{int}(A_i)}$ map geodesics of $P_{\gamma_i}$ to geodesics of $\Hyp^2$. Moreover, by Equation \eqref{eq:ext bdy plane}, $\pi_l^\pm$ maps $\mathrm{int}(\partial_\infty (A_i))=\gr(\gamma_i^{-1}|_{\mathrm{int}(I_i)})$ to $\mathrm{int}(I_i)$. Hence $\pi_l^\pm(\mathrm{int}(A_i))=\mathrm{int}(D_i)$. Analogously, $\pi_r^\pm(\mathrm{int}(A_i))=\gamma_1^{-1}(\mathrm{int}(D_i))=\gamma_2^{-1}(\mathrm{int}(D_i))$.

To see that $\pi_l^\pm$ and $\pi_r^\pm$ are bijective, it remains to show that the image of the geodesic $A_1\cap A_2=P_{\gamma_1}\cap P_{\gamma_2}$ via $\pi_l^\pm$ is the geodesic $\ell=D_1\cap D_2$, while the image via $\pi_r^\pm$ is the geodesic $\gamma_1^{-1}(\ell)=\gamma_2^{-1}(\ell)$. The definition of $\pi_l^\pm$ and $\pi_r^\pm$ on $A_1\cap A_2$ actually depends on the choice of a support plane. Recall that we must choose the \emph{same} support plane at any point $p\in A_1\cap A_2$. From Corollary \ref{cor:all support}, the possible choices of support planes at $p$ are of the form $P_{\sigma\gamma_1}$, where $\sigma$ has the same fixed points as $\gamma_2\circ\gamma_1^{-1}$, which are precisely the common endpoints of $I_1$ and $I_2$.

Using the notation from Lemma \ref{lemma intersection two planes}, we thus see that the endpoints at infinity of $A_1\cap A_2$ are the points $(x,y)$ and $(x',y')$ where $x,x'$ are the fixed points of $\gamma_2\circ\gamma_1^{-1}$ (and of $\sigma$). Hence from Equation \eqref{eq:ext bdy plane} we have (for any choice of $\sigma$ as in the third item of Corollary \ref{cor:all support}) $\pi_l^{\pm}(x,y)=x$ and $\pi_l^{\pm}(x',y')=x'$. Since $\pi_l^\pm$ is, as before, the restriction of an  isometry between $P_{\sigma\gamma_1}$ and $\Hyp^2$, it maps geodesics to geodesics, hence $\pi_l^\pm(A_1\cap A_2)=\ell$. Analogously, $\pi_r^{\pm}(x,y)=y$ and $\pi_r^{\pm}(x',y')=y'$, from which it follows that $\pi_l^\pm(A_1\cap A_2)=\gamma_1^{-1}(\ell)=\gamma_2^{-1}(\ell)$.
This concludes the proof of the first item. 

For the second item, define $E:=\pi_r^-\circ (\pi^-_l)^{-1}$, which is a bijection of $\Hyp^2$. Consider the geodesic lamination of $\Hyp^2$ which is composed by the sole geodesic $\ell$. Hence the strata of $\ell$ are: $\mathrm{int}(D_1),\mathrm{int}(D_2)$ and $\ell$. We will show that the comparison isometries $\mathrm{Comp}(S,S'):=(E|_{S})^{-1}\circ E|_{S'}$ all translate to the right or to the left seen from one stratum to another, according to as $\gamma_2\circ\gamma_1^{-1}$ translates to the left or to the right seen from $D_1$ to $D_2$. 

Let us first consider $S=\mathrm{int}(D_1)$ and $S'=\mathrm{int}(D_2)$. Then by Example \ref{ex: extension isometry} (see in particular Equation \eqref{eq:compose plane}) $E$ equals $\gamma_i^{-1}$ on $\mathrm{int}(D_i)$, because $(\pi_l^\pm)^{-1}(\mathrm{int}(D_i))=\mathrm{int}(A_i)\subset P_{\gamma_i^{-1}}$. Hence the comparison isometry 
$\mathrm{Comp}(\mathrm{int}(D_1),\mathrm{int}(D_2))$ equals $\gamma_1\circ\gamma_2^{-1}$, and it translates to the left (resp. right) seen from $\mathrm{int}(D_1)$ to $\mathrm{int}(D_2)$ exactly when $\gamma_2\circ\gamma_1^{-1}$, which is its inverse, translates to the right (resp. left). The proof when one of the two strata $S$ or $S'$ is $\ell$ is completely analogous, by using the third item of Corollary \ref{cor:all support}. Indeed (recalling Remark \ref{rmk:choice}), by any choice of $\sigma$ of the form $\sigma=\exp(t\mathfrak a)$ with $t\in (0,1)$, $\mathrm{Comp}(\ell,\mathrm{int}(D_2))=\sigma\circ\gamma_2^{-1}$ translates to the left (resp. right) seen from $\ell$ to $\mathrm{int}(D_2)$, and $\mathrm{Comp}(\mathrm{int}(D_1),\ell)=\gamma_1\circ\sigma^{-1}$ translates to the left (resp. right) seen from $\mathrm{int}(D_1)$ to $\ell$. If instead $\sigma=\exp(t\mathfrak a)$ with $t\in\{0,1\}$, then $\sigma$ coincides either with $\gamma_1$ or with $\gamma_2$, which means that one of the comparison isometries $\mathrm{Comp}(\mathrm{int}(D_1),\ell)$ and $\mathrm{Comp}(\ell,\mathrm{int}(D_2))$ translates to the left, and the other is the identity, which is still allowed in the definition of earthquake because $\ell$ is in the boundary of $\mathrm{int}(D_i)$.
\end{proof}

\subsection{The example is prototypical}

The case of simple earthquakes that we have considered above may appear as very special. However, it turns out that it is the prototypical example, that will serve to treat the general case in the proof of the earthquake theorem. The following lemma shows that the situation of two intersecting planes occurs quite often.

\begin{lemma}\label{lemma:intersection of two support planes}
Let $f:\mathbb{RP}^1\to \mathbb{RP}^1$ be an orientation-preserving homeomorphism which is not in $\PSL(2,\R)$. Then:
\begin{enumerate}
\item Any two   future support planes of $\con(f)$ at points of $\partial_+\con(f)$ intersect in $\AdS^3$. Analogously, any two past support planes of $\con(f)$ at points of $\partial_-\con(f)$ intersect in $\AdS^3$.
\item Given a point $p\in \partial_\pm\con(f)$, if there exist two support planes at  $p$, then their intersection (which is a spacelike geodesic) is contained in $\partial_\pm\con(f)$. As a consequence, any other support plane at $p$ contains this spacelike geodesic.
\end{enumerate}
\end{lemma}

\begin{proof}
Let us consider future support planes, the other case being analogous. For the first item, let $P$ and $Q$ be support planes intersecting $\partial_+\con(f)$, which are spacelike by Proposition \ref{prop:support planes space or light}, and suppose by contradiction $P$ and $Q$ that they are disjoint. We can slightly move them in the future   to get spacelike planes $P'$ and $Q'$ such that $P$, $Q$, $P'$ and $Q'$ are mutually disjoint and $P'\cap \partial_+ \con(f)=Q'\cap \partial_+\con(f)=\emptyset$. (For instance, if $P=P_{\gamma_1}$ and $Q=P_{\gamma_2}$, then we can use Lemma \ref{lemma intersection two planes} and consider  $P'=P_{\sigma\gamma_1}$ and $Q'=P_{\sigma\gamma_1}$ for $\sigma$ an elliptic element of small clockwise angle of rotation.)

Now, observe that $\AdS^3\setminus (P'\cup Q')$ is the disjoint union of two cylinders and $P$ and $Q$ lie in different connected components of this complement. See Figure \ref{fig:tore}. However, $\partial_{+}\con(f)$ is connected and has empty intersection with $P$ and $Q$, leading to a contradiction.

For the second item, let $P=P_{\gamma_1}$ and $Q=P_{\gamma_2}$ be support planes such that $p\in\partial_+\con(f)\cap P\cap Q$. By   Lemma \ref{lemma intersection two planes}, $\gamma_2\circ\gamma_1^{-1}$ is hyperbolic. Up to switching the roles of $\gamma_1$ and $\gamma_2$, we can assume that $\gamma_2\circ\gamma_1^{-1}$ translates to the left seen from $D_1$ to $D_2$, where as usual $D_i$ is the convex envelope of the interval $I_i$, and the common endpoints $x,x'$ of $I_1$ and $I_2$ are the fixed points of $\gamma_2\circ\gamma_1^{-1}$. Hence $\partial_\infty P_{\gamma_1}\cap\partial_\infty P_{\gamma_2}=\{(x,y),(x',y')\}$ where $y=\gamma_1^{-1}(x)=\gamma_2^{-1}(x)$ and $y'=\gamma_1^{-1}(x')=\gamma_2^{-1}(x')$.

Now, by \eqref{eq:shape}, $P_{\gamma_i}\cap\gr(f)$ consists of at least two points for $i=1,2$. We claim that $\gr(f)\cap P_{\gamma_i}$ contains at least $(x,y)$ and $(x',y')$. Indeed, since $P_{\gamma_2}$ is a support plane, $\con(f)\cap P_{\gamma_1}$ is contained in the half-plane $A_1\subset P_{\gamma_1}$. If $\gr(f)\cap P_{\gamma_1}$ had not contained $(x,y)$ and $(x',y')$, then $\con(f)\cap P_{\gamma_1}$ would not contain the boundary geodesic $A_1\cap A_2$, and thus would not contain $p$. The same argument applies for $P_{\gamma_2}$. This shows that both $(x,y)$ and $(x',y')$ are in $\gr(f)$, and therefore the spacelike geodesic $P_{\gamma_1}\cap P_{\gamma_2}$ is in $\partial_\pm\con(f)$.
\end{proof}

\begin{figure}[htb]
\centering
\includegraphics[height=6cm]{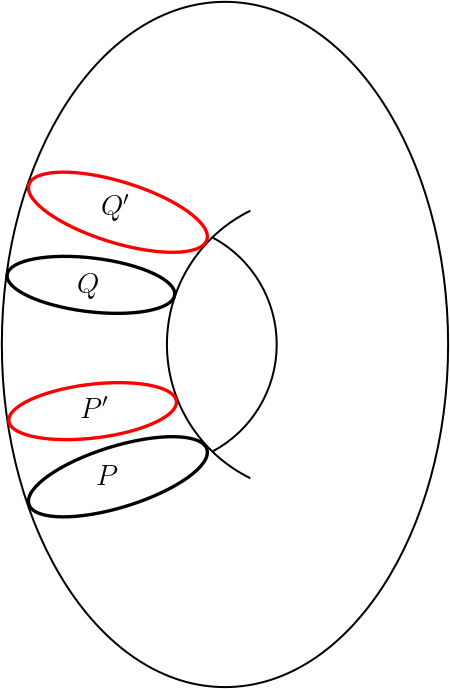}
\caption{The setting of the proof of Lemma \ref{lemma:intersection of two support planes}.} \label{fig:tore}      
\end{figure}

\begin{remark}
In the first item of Lemma \ref{lemma:intersection of two support planes}, the hypothesis that $P$ and $Q$ are support planes at points of $\partial^\pm\con(f)$ (hence not at points of $\gr(f)\subset\partial_\infty\AdS^3$) is necessary. Recall that by 
Proposition \ref{prop:support planes space or light} support planes of $\con(f)$ are either spacelike or lightlike, and they are necessarily spacelike if they intersect $\con(f)$ at points of $\partial_\pm\con(f)$. 

Now, if one of the two planes $P$ or $Q$ is a support plane at a point of $\gr(f)$, then the proof only shows that $P$ and $Q$ must intersect in $\overline\AdS{}^3$, but not necessarily in the interior. It can perfectly happen that two future (or two past) support planes (one of which possibly lightlike) at a point $(x,f(x))$ of $\gr(f)$ intersect at $(x,f(x))$ but not in the interior of $\AdS^3$.
\end{remark}

Lemma \ref{lemma:intersection of two support planes} has an important consequence. Recall that the definition of the projections $\pi_l^\pm,\pi_r^\pm:\partial_\pm\con(f)\to\Hyp^2$ depends on the choice of a support plane at all points $p$ that admit more than one support plane. Moreover, we require that this support plane is chosen to be constant on any connected component of the subset of $\partial_\pm\con(f)$ consisting of those points that admit more than one support plane (Remark \ref{rmk:choice}). We will now see that, roughly speaking, their image does \emph{not} depend on this choice of support plane.

\begin{corollary}\label{cor: image does not depend}
Let $f:\mathbb{RP}^1\to \mathbb{RP}^1$ be an orientation-preserving homeomorphism which is not in $\PSL(2,\R)$, and suppose $p\in\partial_\pm\con(f)$ has at least two support planes. Then there exist $\gamma_1,\gamma_2\in\PSL(2,\R)$ with $\gamma_2\circ\gamma_1^{-1}=\exp(\mathfrak a)$ a hyperbolic element, such that all support planes at $p$  are precisely those of the form $P_{\sigma\gamma_1}$ where $\sigma=\exp(t\mathfrak a)$ for $t\in [0,1]$. The same conclusion holds for all other points $p'\in P_{\gamma_1}\cap P_{\gamma_2}$.

In particular, the image of the spacelike geodesic $P_{\gamma_1}\cap P_{\gamma_2}$ under the projections $\pi_l^\pm,\pi_r^\pm:\partial_\pm\con(f)\to\Hyp^2$ is a geodesic in $\Hyp^2$ that does not depend on the choice of the support plane as in the definition of $\pi_l^\pm$ and $\pi_r^\pm$.
\end{corollary}
\begin{proof}
Suppose $P_{\hat\gamma_1}$ and $P_{\hat\gamma_2}$ are (say, future) distinct support planes at $p$. Write $\hat\gamma_2\circ\hat\gamma_1^{-1}=\exp(\hat{\mathfrak a})$, which is a hyperbolic element by Lemma \ref{lemma intersection two planes} and the first item of Lemma \ref{lemma:intersection of two support planes}. By the second item of Lemma \ref{lemma:intersection of two support planes}, any other support plane at $p$ must be of the form $P_{\sigma\hat\gamma_1}$ for $\sigma$ an element having the same fixed points as $\hat\gamma_2\circ\hat\gamma_1^{-1}$. That is, $\sigma$ is of the form $\exp(s\hat{\mathfrak a})$ for some $s\in\R$.

We claim that the set 
$$I=\{s\in\R\,|\,\exp(s\hat{\mathfrak a})\text{ is a support plane of }\con(f)\text{ at }p\}$$
is a compact interval. This will conclude the proof, up to applying an affine change of variable mapping the interval $I=[s_1,s_2]$ to $[0,1]$, and defining $\gamma_i=\exp(s_i\hat{\mathfrak a})$.

To prove the claim, suppose that $s,s'\in I$. Then $\con(f)$ is contained in the past of a pleated surface obtained as the union of two half-spaces, one contained in $P_{\exp(s\hat{\mathfrak a})\hat\gamma_1}$ and the other in $P_{\exp(s'\hat{\mathfrak a})\hat\gamma_1}$, meeting along the spacelike geodesic $P_{\hat\gamma_1}\cap P_{\hat\gamma_2}$. Then every support plane for this pleated surface is a support plane for $\con(f)$ as well. That is, by the last item of Corollary \ref{cor:all support}, $[s,s']\subset I$. This shows that $I$ is an interval. It is compact by Lemma \ref{lem: convergence support planes}, applied to the constant sequence $p_n=p$ and to $\gamma_n=\exp(s_n\hat{\mathfrak a})\hat\gamma_1$, showing that $s_n$ must be converging (up to subsequences) and its limit is in $I$. This concludes the proof.
\end{proof}


\section{Proof of the earthquake theorem}\label{sec:proof}

We are now ready to enter into the details of the proof of the earthquake theorem. The outline of the proof is now clear: given an orientation-preserving homeomorphism $f:\mathbb{RP}^1\to \mathbb{RP}^1$ (which we can assume is not in $\PSL(2,\R)$), we consider the projections $\pi_l^\pm,\pi_r^\pm:\partial_\pm\con(f)\to\Hyp^2$, and we want to show that the composition $\pi_r^\pm\circ (\pi_l^\pm)^{-1}$ is well-defined and is a (left or right) earthquake map extending $f$. We will prove this in several steps: the proof of Theorem \ref{thm:main} will follow from Proposition \ref{prop:bijective}, Corollary \ref{cor:bijective+extension} and Proposition \ref{prop:earthquake} below.

\subsection{Extension to the boundary}

The first property we study is the extension of the projections $\pi_l^\pm$ and $\pi_r^\pm$ to the boundary.

\begin{proposition}\label{prop:extension}
The projections $\pi_l^\pm,\pi_r^\pm:\partial_\pm\con(f)\to\Hyp^2$ extend to $\gr(f)$. More precisely, if $p_n\in\partial_\pm\con(f)\to(x,y)\in\gr(f)$, then $\pi_l^\pm(p_n)\to x$ and $\pi_r^\pm(p_n)\to y$. 
\end{proposition}
Observe that the conclusion of Proposition \ref{prop:extension} holds for \emph{any} choice of the projections $\pi_l^\pm$ and $\pi_r^\pm$, regardless of the chosen support planes when several choices are possible, as in Remark \ref{rmk:choice}. The proof involves two well-known properties of isometries in plane hyperbolic geometry; for the sake of completeness, we provide elementary, self-contained proofs in the Appendix.

\begin{proof}
Let $p_n\in\partial_\pm\con(f)$ be a sequence converging to $(x,y)\in\gr(f)$, and let $P_{\gamma_n}$ be a sequence of support planes of $\con(f)$ at $p_n$, which are necessarily spacelike by Proposition \ref{prop:support planes space or light}. By Lemma \ref{lem: convergence support planes}, up to extracting a subsequence, there are two possibilities: either $\gamma_n\to\gamma$ and $P_{\gamma_n}$ converges to the spacelike support plane $P_\gamma$, or $\gamma_n$ diverges in $\PSL(2,\R)$ and $P_{\gamma_n}$ converges to the lightlike plane whose boundary is $(\{x\}\times\RP^1)\cup(\RP^1\cup\{y\})$. We will treat these two situations separately, and we will always use the characterization of the convergence to the boundary given in \eqref{eq:convergence}.

Consider the former case, namely when $\gamma_n\to\gamma$. We have by hypothesis that 
\begin{equation}\label{eq:hypothesis convergence}
p_n(z_0)\to x\qquad\text{and}\qquad p_n^{-1}(z_0)\to y~,
\end{equation}
for any point $z_0\in\Hyp^2$. Observe moreover that, from the definition of the projections,
\begin{equation}\label{eq:defi proj2}
\pi_l^\pm(p_n)=\Fix(p_n\gamma_n^{-1})\qquad\text{and}\qquad\pi_r^\pm(p_n)=\Fix(\gamma_n^{-1}p_n)~.
\end{equation}
 Recalling (see \eqref{eq:graph id}) that the boundary of $P_\1$ is identified with $\RP^1$ via the map $(x,x)\mapsto x$, we thus have to show (choosing for instance the point $z_0=i$) that:  $p_n\gamma_n^{-1}(i)\to x$ and $\gamma_n^{-1} p_n(i)\to y$. However, since $\gamma_n\to\gamma$, $p_n\gamma_n^{-1}(i)$ is at bounded distance from $p_n\gamma^{-1}(i)$. Applying the hypothesis \eqref{eq:hypothesis convergence} to $z_0=\gamma^{-1}(i)$, we have $p_n\gamma^{-1}(i)\to x$ and therefore $p_n\gamma_n^{-1}(i)\to x$. 
The argument is analogous to show that $\gamma_n^{-1} p_n(i)\to y$, except that it is useful to observe that $\gamma_n^{-1} p_n=p_n^{-1}\gamma_n$ since it is an order--two isometry. Now $p_n^{-1}\gamma_n(i)$ is at bounded distance from $p_n^{-1}\gamma(i)$, which converges to $y$ by hypothesis. Hence $p_n^{-1}\gamma_n(i)\to y$ and the proof is complete for this case. 

Let us move on to the latter case, that is, $\gamma_n$ diverges in $\PSL(2,\R)$. Here we must use not only the previous assumption \eqref{eq:hypothesis convergence}, but also the following:
\begin{equation}\label{eq:hypothesis convergence2}
\gamma_n(z_0)\to x\qquad\text{and}\qquad \gamma_n^{-1}(z_0)\to y~,
\end{equation}
for any $z_0\in\Hyp^2$. The condition \eqref{eq:hypothesis convergence2} holds
because $\gamma_n$ converges to the projective class of a rank one matrix $A$, such that $P_{[A]}$ is a lightlike support plane; we have already observed that the boundary at infinity of $P_{[A]}$ must be equal to $(\{x\}\times\RP^1)\cup(\RP^1\cup\{y\})$. Combining \eqref{eq:bdy}, \eqref{eq:convergence} and Lemma \ref{lem:lightlike planes}, we deduce that $\gamma_n(z_0)\to x$ and $\gamma_n^{-1}(z_0)\to y$ as claimed.

Having made this preliminary observation, now we can rewrite \eqref{eq:defi proj2} as the identities:
\begin{equation}\label{eq:defi proj3}
p_n= \mathcal R_{\pi_l^\pm(p_n)}\circ\gamma_n\qquad\text{and}\qquad p_n^{-1}= \mathcal R_{\pi_r^\pm(p_n)}\circ\gamma_n^{-1}~,
\end{equation}
 where we recall that $\mathcal R_w$ denotes the order two elliptic isometry with fixed point $w\in\Hyp^2$. Up to extracting a subsequence, we can assume that $\pi_l^\pm(p_n)\to \hat x_\pm$ and $\pi_r^\pm(p_n)\to \hat y_\pm$, for some points $\hat x_\pm,\hat y_\pm\in\Hyp^2\cup\partial_\infty\Hyp^2$. We need to show that $\hat x_\pm=x$ and $\hat y_\pm=y$. 

For this purpose, suppose by contradiction $\hat x_\pm\neq x$. Suppose first that $\hat x_\pm\in\Hyp^2$. We will use the fact (Lemma \ref{lemma:unif cont} in the Appendix) that if $w_n\to w\in \Hyp^2$, then $\mathcal R_{w_n}$ converges to $\mathcal R_w$ uniformly on $\Hyp^2\cup\partial_\infty\Hyp^2$. From \eqref{eq:defi proj3}, and the fact that, from \eqref{eq:hypothesis convergence} and \eqref{eq:hypothesis convergence2}, both $p_n(z_0)$ and $\gamma_n(z_0)$ converge to $x$, we would then have 
$$x=\lim_n p_n(z_0)=\lim_n (\mathcal R_{\pi_l^\pm(p_n)}(\gamma_n(z_0)))=\mathcal R_{\hat x_\pm}(x)\neq x$$
since $\mathcal R_{\hat x_\pm}$ does not have fixed points on $\partial_\infty\Hyp^2$, thus giving a contradiction. If $\hat y_\pm\in\Hyp^2$, we get a contradiction by an analogous argument.

Finally, if $\hat x_\pm\in \partial_\infty\Hyp^2$, we can find a neighbourhood $U$ of $\hat x_\pm$ not containing $x$, such that for $n$ large $\mathcal R_{\pi_l^\pm(p_n)}$ maps the complement of $U$ inside $U$ (see Lemma \ref{lemma:div rotations} in the Appendix). This gives a contradiction with \eqref{eq:defi proj3} because $p_n(z_0)$ and $\gamma_n(z_0)$ are in the complement of $U$ for $n$ large, but at the same time $\mathcal R_{\pi_l^\pm(p_n)}(\gamma_n(z_0))$ should be in $U$ for $n$ large. The argument for $\hat y_\pm$ is completely analogous.
\end{proof}

\begin{remark}
We remark that the proof of Proposition \ref{prop:extension} does not use the full hypothesis that the surface on which the projections are defined is a boundary component of $\con(f)$, but only the property that whenever a sequence $P_{\gamma_n}$ of spacelike support planes converges to a lightlike plane, then this limit is a support plane too, which is true for any convex surface.
\end{remark}

\subsection{Invertibility of the projections}

The next step in the proof is to show that the projections $\pi_l^\pm$ and $\pi_r^\pm$ are bijective.

\begin{proposition}\label{prop:bijective}
The projections $\pi_l^\pm,\pi_r^\pm:\partial_\pm\con(f)\to\Hyp^2$ are bijective.
\end{proposition}
\begin{proof}
We give the proof for $\pi_l^\pm$, the proof for $\pi_r^\pm$ being completely identical. Let us first show that $\pi_l^\pm$ and $\pi_r^\pm$ are injective. Given $p_1,p_2\in\partial_\pm\con(f)$, let $P_{\gamma_1}$ and $P_{\gamma_2}$ be the support planes at $p_1$ and $p_2$ respectively. (If there are several support planes, we choose one, as in the definition of $\pi_l^\pm$ and $\pi_r^\pm$ --- see Remark \ref{rmk:choice}.) By Lemma \ref{lemma intersection two planes} and Lemma \ref{lemma:intersection of two support planes}, $\gamma_2\circ\gamma_1^{-1}$ is a hyperbolic isometry; let $D_1$ and $D_2$ be the convex envelopes in $\Hyp^2$ of the two intervals $I_1$ and $I_2$ with endpoints the fixed points of $\gamma_2\circ\gamma_1^{-1}$. Up to switching $\gamma_1$ and $\gamma_2$, we can moreover assume that $\gamma_2\circ\gamma_1^{-1}$ translates to the left seen from $D_1$ to $D_2$.

Now, we will use the example studied in Section \ref{sec:example}. Let $f^+_{\gamma_1,\gamma_2}$ be defined as in \eqref{eq:f+-}. By Corollary \ref{cor:all support}, $P_{\gamma_i}$ is the support plane of $\con(f^+_{\gamma_1,\gamma_2})$ at the point $p_i\in\partial_\pm\con(f^+_{\gamma_1,\gamma_2})$, for $i=1,2$. Hence $\pi_l^\pm(p_i)=\hat\pi_l^\pm(p_i)$, where    $\hat\pi_l^\pm$ is the left projection associated with $\con(f^+_{\gamma_1,\gamma_2})$. Since $\hat\pi_l^\pm(p_i)$ is bijective by  Proposition \ref{prop:proj simple}, $\pi_l^\pm(p_1)\neq \pi_l^\pm(p_2)$. This shows the injectivity.

To prove the surjectivity, we first show that the image is closed. Suppose $z_n=\pi_l^\pm(p_n)$ is a sequence in the image, with $\lim z_n=z\in\Hyp^2$. Up to extracting a subsequence, we can assume $p_n\to p\in \partial_\pm\con(f)\cup\gr(f)$. From Proposition \ref{prop:extension}, we have that $p\in\partial_\pm\con(f)$, because if $p=(x,y)\in\gr(f)$, then $\pi_l^\pm(p_n)\to x\in\partial_\infty\Hyp^2$, thus contradicting the hypothesis $z_n\to z\in\Hyp^2$. Now, let $P_{\gamma_n}$ be a support plane at $p_n$, which is spacelike by Proposition \ref{prop:support planes space or light}. By Lemma \ref{lem: convergence support planes}, up to extracting a subsequence, $\gamma_n\to\gamma\in\PSL(2,\R)$ and $P_{\gamma}$ is a spacelike support plane at $p$. It is important to remark that $\partial_\pm\con(f)$ might admit several support planes at $p$, and $P_{\gamma}$ might not be the support plane that has been chosen in the definition of $\pi_l^\pm$; however, by Corollary \ref{cor: image does not depend} the image does not depend on this choice.  Hence we can assume that $P_{\gamma}$ \emph{is} the support plane chosen at $p$. That is, from \eqref{defi projections}, $\pi_l^\pm(p)=\Fix(p\circ\gamma^{-1})$. We can now conclude that $z$ is in the image of $\pi_l^\pm$: on the one hand $z_n=\pi_l^\pm(p_n)=\Fix(p_n\circ\gamma_n^{-1})$ converges to $z$ by hypothesis, and on the other it converges to $\pi_l^\pm(p)=\Fix(p\circ\gamma^{-1})$ because $p_n\to p$, $\gamma_n\to\gamma$ and $\Fix$ is continuous. This shows that $z\in\pi_l^\pm(\partial_\pm\con(f))$, and therefore the image is closed.

We now proceed to show that $\pi_l^\pm$ is surjective. Suppose by contradiction that there is a point $w\in\Hyp^2$ which is not in the image of $\pi_l^\pm$. Let $r_0=\inf\{r\,|\,B(w,r)\cap \pi_l^\pm(\partial_\pm\con(f))\neq\emptyset\}$, where $B(w,r)$ is the open ball centered at $w$ of radius $r$ with respect to the hyperbolic metric of $\Hyp^2$. Since the image of $\pi_l^\pm$ is closed, we have that $r_0>0$, $B(w,r_0)$ is disjoint from the image of $\pi_l^\pm$, and there exists a point $z\in\partial B(w,r_0)$ which is in the image of $\pi_l^\pm$. Say that $z=\pi_l^\pm(p)$. We will obtain a contradiction by finding points close to $p$ which are mapped by $\pi_l^\pm$ inside $B(w,r_0)$.

Let $P_\gamma$ be a support plane of $\con(f)$ at $p$. By \eqref{eq:shape}, $P_{\gamma}\cap\con(f)$ is the convex hull of $\partial_\infty P_\gamma\cap\gr(f)$, which contains at least two points. If $p$ is in the interior of $P_{\gamma}\cap\con(f)$ (which is non-empty if and only if $\partial_\infty P_\gamma\cap\gr(f)$ contains at least three points), then the restriction of $\pi_l^\pm$ to the interior of $P_{\gamma}\cap\con(f)$ is an isometry onto its image in $\Hyp^2$, because $P_{\gamma}$ is the unique support plane at interior points $p'$, and  $\pi_l^\pm(p')=\Fix(p'\circ\gamma^{-1})$. Hence $\pi_l^\pm$ maps a small neighbourhood of $p$ to a neighbourhood of $z$, which intersects $B(w,r_0)$, giving a contradiction.

We are only left with the case where $p$ is not in the interior of $P_{\gamma}\cap\con(f)$. In this case, there is a geodesic $L$ contained in $P_{\gamma}\cap\con(f)$ such that $p\in L$. (The geodesic $L$ might be equal to $P_{\gamma}\cap\con(f)$ or not.) As before, the image of $L$ is a geodesic $\ell$ in $\Hyp^2$ because $(\pi_l^\pm)|_L$ is an isometry onto its image, and $z\in \ell$. We claim that in the image of $\pi_l^\pm$ there are two sequences of geodesics $\ell_n\subset\mathrm{Im}(\pi_l^\pm)$ converging to $\ell$ (in other words, such that the endpoints of $\ell_n$ converge to the endpoint of $\ell$); moreover the two sequences are contained in different connected components of $\Hyp^2\setminus\ell$. This will give a contradiction, because for one of these two sequences, $\ell_n$ must intersect $B(w,r_0)$ for $n$ large.

To show the claim, and thus conclude the proof, observe that $L$ disconnects $\partial_\pm\con(f)$ in two connected components, and let $p_n$ be a sequence converging to $p$ contained in one connected component of $\partial_\pm\con(f)\setminus L$. Let $P_{\gamma_n}$ be the support plane for $\con(f)$ at $p_n$ which has been chosen to define $\pi_l^\pm$. By Lemma \ref{lem: convergence support planes}, $P_{\gamma_n}$ converges to a support plane $P_{\gamma}$ at $p$, which as before we can assume is the support plane that defined $\pi_l^\pm$ at $p$, since the image does not depend on this choice by Corollary \ref{cor: image does not depend}. Also, we can assume that each $p_n$ is contained in a geodesic $L_n$ in $P_{\gamma_n}\cap\partial_\pm\con(f)$: indeed, it suffices to replace $p_n$ by the point in $P_{\gamma_n}\cap\partial_\pm\con(f)$ which is closest to $p$ (where closest is with respect to the induced metric on $\partial_\pm\con(f)$, or to any auxiliary Riemannian metric). If $P_{\gamma_n}\cap\partial_\pm\con(f)$ is not already a geodesic, with this assumption $p_n$ now belongs to a boundary component which is the geodesic $L_n$. As observed before, $\pi_l^\pm$ maps the geodesic $L_n$ to a geodesic $\ell_n=\pi_l^\pm(L_n)$ in $\Hyp^2$, and (as in the argument that showed that $\mathrm{Im}(\pi_l^\pm)$ is closed), the limit of $\pi_l^\pm(p_n)$ is a point in $\ell=\pi_l^\pm(L)$. 

Moreover $\ell_n\cap\ell=\emptyset$, and the $\ell_n$ are all contained in the same 
connected component of $\Hyp^2\setminus\ell$: this follows from observing again (compare with the injectivity at the beginning of this proof) that $(\pi_l^\pm)|_{L_n\cup L}$ equals the left projection associated with the surface $\partial_\pm\con(f^+_{\gamma_n,\gamma})$ studied in Section \ref{sec:example}, where $f^+_{\gamma_1,\gamma_2}$ is defined in \eqref{eq:f+-}, and thus maps $\partial_\pm\con(f)\cap P_{\gamma_n}$ (which in particular contains $L_n$) to a subset (containing $\ell_n$) disjoint from $\ell$ and included in a connected component of $\Hyp^2\setminus \ell$ which does not depend on $n$.

This implies that $\ell_n$ converges to $\ell$ as $n\to+\infty$. Clearly if we had chosen $p_n$ in the other connected component of $\partial_\pm\con(f)\setminus L$, then the $\ell_n$ would be contained in the other connected component of $\Hyp^2\setminus \ell$. This concludes the claim and thus the proof.
\end{proof}

As a consequence, the composition $\pi_r^\pm\circ (\pi_l^\pm)^{-1}$ is well-defined and is a bijection of $\Hyp^2$ to itself. Combining with 
Proposition \ref{prop:extension}, we get:

\begin{corollary}\label{cor:bijective+extension}
The composition $\pi_r^\pm\circ (\pi_l^\pm)^{-1}$ extends to a bijection from $\Hyp^2\cup\partial_\infty\Hyp^2$ to itself, which equals $f$ on $\partial_\infty\Hyp^2$ and is continuous at any point of $\partial_\infty\Hyp^2$.
\end{corollary}
\begin{proof}
Since $\pi_l^\pm$ and $\pi_r^\pm$ are bijective and extend to the bijections from $\gr(f)$ to $\partial_\infty\Hyp^2$ sending  $(x,y)$ to $x$ and $y=f(x)$ respectively, the composition  $\pi_r^\pm\circ (\pi_l^\pm)^{-1}$ extends to a bijection of $\Hyp^2\cup \partial_\infty\Hyp^2$ to itself sending $x$ to $f(x)$. 

We need to check that this extension is continuous at any point of $\partial_\infty\Hyp^2$. Proposition \ref{prop:extension} shows that the extensions of $\pi_l^\pm$ and $\pi_r^\pm$ to $\partial_\pm\con(f)\cup\gr(f)$ are continuous at any point of $\gr(f)$. 
Hence it remains to check that $(\pi_l^\pm)^{-1}$ is continuous at any point of $\partial_\infty\Hyp^2$. 

This follows from a standard argument: let $z_n$ be a sequence in $\Hyp^2\cup\partial_\infty\Hyp^2$ converging to $x\in \partial_\infty\Hyp^2$, and let $p_n=(\pi_l^\pm)^{-1}(z_n)$. Up to extracting a subsequence, $p_n\to p$. The limit $p$ must be in $\gr(f)$, because if $p\in \partial_\pm\con(f)$, although $\pi_l^\pm$ might not be continuous there, we have already seen in Proposition \ref{prop:bijective} (see the proof that the image of $\pi_l^\pm$ is closed) that $\lim_n\pi_l^\pm(p_n)=\lim_n z_n$ is a point of $\Hyp^2$, thus giving a contradiction with $\lim_n z_n=x\in \partial_\infty\Hyp^2$. If $p\in\gr(f)$, then we can use the continuity and injectivity  of $\pi_l^\pm$ on $\gr(f)$ to infer that $p=(\pi_l^\pm)^{-1}(x)$. This concludes the proof.
\end{proof}

\subsection{Earthquake properties}

The last step which is left to prove is the verification that $\pi_r^\pm\circ (\pi_l^\pm)^{-1}$ satisfies the properties defining earthquake maps. 

\begin{proposition}\label{prop:earthquake}
The composition $\pi_r^-\circ (\pi_l^-)^{-1}:\Hyp^2\to\Hyp^2$ is a left earthquake map. Analogously, $\pi_r^+\circ (\pi_l^+)^{-1}:\Hyp^2\to\Hyp^2$ is a right earthquake map. 
\end{proposition}
\begin{proof}
First, let us define a geodesic lamination $\lambda$. Let us consider all the support planes $P_\gamma$  of $\con(f)$ at points of $\partial_\pm\con(f)$ (which are necessarily spacelike by Proposition \ref{prop:support planes space or light}). Define $\mathcal L$ to be the collection of all the connected components of $(P_\gamma\cap \partial_\pm\con(f))\setminus\mathrm{int}(P_\gamma\cap \partial_\pm\con(f))$, as $P_{\gamma}$ varies over all support planes. As observed before, by \eqref{eq:shape} $P_\gamma\cap \partial_\pm\con(f)$ is the convex hull in $P_\gamma$ of  $\partial_\infty P_\gamma\cap\gr(f)$, which consists of at least two points. If it consists of exactly two points, then  $P_\gamma\cap \partial_\pm\con(f)$ is a spacelike geodesic $L$; otherwise $P_\gamma\cap \partial_\pm\con(f)$ has nonempty interior and each connected component of its boundary is a spacelike geodesic. Now, $\pi_l^\pm$ is an isometry onto its image when restricted to any $L\in\mathcal L$ (which might depend on the choice of a support plane if there are several support planes at points of $L$, but the image does not depend on this choice by Corollary \ref{cor: image does not depend}). Hence we define $\lambda$ to be the collection of all the $\pi_l^\pm(L)$ as $L$ varies in $\mathcal L$. 

To show that $\lambda$ is a geodesic lamination, we first observe that the geodesics $\ell\in\lambda$ are pairwise disjoint, because the spacelike geodesics $L$ in $\mathcal L$ are pairwise disjoint and $\pi_l^\pm$ is injective. Then it remains to show that their union is a closed subset of $\Hyp^2$. This follows immediately from  the proof of Proposition \ref{prop:bijective}. Indeed, suppose that $\ell_n=\pi_l^\pm(L_n)$ converges to $\ell=\pi_l^\pm(L)$, and let $z_n=\pi_l^\pm(p_n)\in\ell_n$ be a sequence converging to $z\in\ell$.
Since $\mathrm{Im}(\pi_l^\pm)$ is closed, $z\in\mathrm{Im}(\pi_l^\pm)$,  and since $\pi_l^\pm$ is injective, $z=\pi_l^\pm(p)$ for some $p\in L$. Then in the last part of the proof of Proposition \ref{prop:bijective} we have shown that in this situation $\ell_n$ converges to $\ell$. 

Having shown that $\lambda$ is a geodesic lamination, we are ready to check that  $\pi_r^-\circ (\pi_l^-)^{-1}$ is an earthquake map. Observe that the gaps of $\lambda$ are precisely the images under $\pi_l^\pm$ of the interior of the sets $P_\gamma\cap \partial_\pm\con(f)$ (when this intersection is not reduced to a geodesic), as $P_{\gamma}$ varies among all support planes. 

Let $S_1$ and $S_2$ be two strata of $\lambda$, and let $\Sigma_i=(\pi_l^\pm)^{-1}(S_i)$. Hence $\Sigma_i\subset P_{\gamma_i}\cap \partial_\pm\con(f)$, where $P_{\gamma_i}$ is a  support plane. As usual, there might be several support planes at points of $\Sigma_i$, and this can occur only if $\Sigma_i$ is reduced to a geodesic by Lemma \ref{lemma:intersection of two support planes}. Recalling from Remark \ref{rmk:choice} that the chosen support plane is assumed to be constant along $\Sigma_i$, we can  suppose that $P_{\gamma_i}$ is the support plane  chosen in the definition of $\pi_l^\pm$ and $\pi_r^\pm$. 

Now we proceed as in the proof of injectivity in Proposition \ref{prop:bijective}. Consider first the case that $\gamma_1\neq\gamma_2$. By Lemma \ref{lemma intersection two planes} and Lemma \ref{lemma:intersection of two support planes}, $\gamma_2\circ\gamma_1^{-1}$ is a hyperbolic isometry; let $D_1$ and $D_2$ be the convex envelopes in $\Hyp^2$ of the two intervals $I_1$ and $I_2$ with endpoints the fixed points of $\gamma_2\circ\gamma_1^{-1}$. Up to switching $\gamma_1$ and $\gamma_2$, we assume that $\gamma_2\circ\gamma_1^{-1}$ translates to the left seen from $D_1$ to $D_2$. Then $(\pi_l^\pm)|_{\Sigma_i}=(\hat\pi_l^\pm)|_{\Sigma_i}$ and $(\pi_r^\pm)|_{\Sigma_i}=(\hat\pi_r^\pm)|_{\Sigma_i}$, where    $\hat\pi_l^\pm$ and $\hat\pi_r^\pm$ are the left and right projections associated with $\con(f^+_{\gamma_1,\gamma_2})$, and moreover $S_i\subset D_i$. By the second part of Proposition \ref{prop:proj simple}, the comparison isometry $\widehat{\mathrm{Comp}}(D_1,D_2)$ of the map $\hat\pi_r^\pm\circ (\hat\pi_l^\pm)^{-1}$ translates to the left (for $\pi_r^-$ and $\pi_l^-$) or right (for $\pi_r^+$ and $\pi_l^+$) seen from $D_1$ to $D_2$. Then ${\mathrm{Comp}}(S_1,S_2)$, which is indeed equal to $\widehat{\mathrm{Comp}}(D_1,D_2)$, translates to the left (or right) seen from $S_1$ to $S_2$. 

Finally, we instead consider the case  $\gamma_1=\gamma_2$, which can only happen either if $\Sigma_1=\Sigma_2$ (hence $S_1=S_2$) or if $\Sigma_1$ has nonempty interior and $\Sigma_2$ is one of its boundary components (or vice versa). In this case we clearly have ${\mathrm{Comp}}(S_1,S_2)=\mathrm{id}$. But the comparison isometry is indeed  allowed in Definition \ref{defi:eart} to be the identity, when one of the two strata is contained in the closure of the other. This concludes the proof.
\end{proof}

The proof of Thurston's earthquake theorem (Theorem \ref{thm:main}) is thus complete.

\subsection{Recovering earthquakes of closed surfaces}\label{subsec:closed}

In this final section, we  recover (Corollary \ref{cor:invariance}) the existence of earthquake maps between two homeomorphic closed hyperbolic surfaces.

Given  a group $G$ and two representations $\rho,\varrho:G\to\PSL(2,\R)$, we say that a map $F$ from $\Hyp^2$ (or $\partial_\infty\Hyp^2$) to itself is $(\rho,\varrho)$-equivariant if it satisfies
$$F\circ\rho(g)=\varrho(g)\circ F
$$
for every $g\in G$.

\begin{corollary}\label{cor:invariance}
Let $S$ be a closed oriented surface and let $\rho,\varrho:\pi_1(S)\to\PSL(2,\R)$ be two Fuchsian representations. Then there exists a $(\rho,\varrho)$-equivariant left earthquake map of $\Hyp^2$, and a $(\rho,\varrho)$-equivariant right earthquake map.
\end{corollary}

\begin{proof}
Let $f:\partial_\infty\Hyp^2\to\partial_\infty\Hyp^2$ be the unique $(\rho,\varrho)$-equivariant orientation-preserving homeomorphism. We claim that there exists a left (resp. right) earthquake as in Theorem \ref{thm:main}, which is itself  $(\rho,\varrho)$-equivariant. For this purpose, observe that for any $g\in\pi_1(S)$, the pair $(\rho(g),\varrho(g))\in\PSL(2,\R)\times\PSL(2,\R)$ acts on $\partial_\infty\AdS^3$ preserving $\gr(f)$, since by \eqref{eq:action graph} and the definition of  $(\rho,\varrho)$-equivariant,  
$$(\rho(g),\varrho(g))\cdot\gr(f)=\gr(\varrho(g)\circ f\circ\rho^{-1}(g))=\gr(f)~.$$
 Hence the convex hull $\con(f)$ is preserved by the action of $(\rho(g),\varrho(g))$ for all $g\in\pi_1(S)$. 

To conclude the proof, we need to show that we can choose support planes at every point of both boundary components of $\con(f)\setminus\gr(f)$ in such a way that this choice of support planes is also preserved by the action of $(\rho(g),\varrho(g))$ for all $g\in\pi_1(S)$. (Clearly it suffices to consider the situation at points that admit more than one support plane, because if $p\in\partial_\pm\con(f)$ has a unique support plane $P$, then $(\rho(g),\varrho(g))\cdot P$ is the unique support plane at $(\rho(g),\varrho(g))\cdot p$.)

When we have shown this, we will take the left and right projections $\pi_l^\pm,\pi_r^\pm$ defined via this invariant choice of support planes. By Lemma \ref{lem:left right proj}, we will then deduce that the left projection $\pi_l^\pm:\partial_\pm\con(f)\to\Hyp^2$ is equivariant with respect to the action of $(\rho(g),\varrho(g))$ on $\partial_\pm \con(f)$ and the action of $\rho(g)$ on $\Hyp^2$; analogously the right projection $\pi_r^\pm:\partial_\pm\con(f)\to\Hyp^2$ is equivariant with respect to the action of $(\rho(g),\varrho(g))$ on $\partial_\pm \con(f)$ and the action of $\varrho(g)$ on $\Hyp^2$. Following the proof of Theorem \ref{thm:main}, the left and right earthquake maps obtained as the composition $(\pi_r^\mp)^{-1}\circ\pi_l^\mp$ will be $(\rho,\varrho)$-equivariant, and the proof will be concluded.

First, we need to prove an intermediate claim. Suppose $p\in\partial_\pm\con(f)$ admits several support planes. By Lemma \ref{lemma:intersection of two support planes}, there is a spacelike geodesic $L\subset\partial_\pm\con(f)$ containing $p$. Let $g\in\pi_1(S)$ be such that  $(\rho(g),\varrho(g))\cdot L=L$. Then we claim that $(\rho(g),\varrho(g))$ maps every support plane at $p$ to itself. To prove this claim, we use Corollary \ref{cor: image does not depend} and suppose up to an isometry (so that, in the notation of Corollary \ref{cor: image does not depend}, $\gamma_1=\1$) that all the support planes at $p$ are of the form $P_{\exp(t\mathfrak a)}$ with $t\in [0,1]$, where $\gamma:=\exp(\mathfrak a)$ is a hyperbolic element.

Clearly $(\rho(g),\varrho(g))$ must preserve the pair of "extreme" support planes $P_\1$ and $P_\gamma$. Hence there are two possibilities: either $(\rho(g),\varrho(g))$ maps $\1$ to $\1$ and $\gamma$ to $\gamma$, or it switches $\1$ and $\gamma$. However, the latter possibility cannot be realized, since the identities $\rho(g)\varrho(g)^{-1}=\gamma$ and $\rho(g)\gamma\varrho(g)^{-1}=\1$ would imply that $\gamma$ has order two, and this is not possible for a hyperbolic element. We thus have $(\rho(g),\varrho(g))\cdot\1=\1$ and $(\rho(g),\varrho(g))\cdot\gamma=\gamma$. This implies first that $\rho(g)=\varrho(g)$. Moreover $\rho(g)\gamma\rho(g)^{-1}=\gamma$, which shows that $\rho(g)=\varrho(g)=\exp(s\mathfrak a)$ for some $s\in\R$. Therefore $\rho(g)\exp(t\mathfrak a)\rho(g)^{-1}=\exp(t\mathfrak a)$ for all $t$, that is $(\rho(g),\varrho(g))=(\rho(g),\rho(g))$ maps every support plane $P_{\exp(t\mathfrak a)}$ to itself.  

Having shown the claim, we can conclude as follows. Observe that the set of points $p\in\partial_\pm\con(f)$ that admit several support planes form a disjoint union of spacelike geodesics in $\partial_\pm\con(f)$, and that this set (say $X$) is invariant under the action of $(\rho(g),\varrho(g))$ for all $g\in\pi_1(S)$. Pick a subset $\{L_i\}_{i\in I}$ of this family of geodesics such that its $\pi_1(S)$-orbit is $X$, and that the orbits of $L_i$ and $L_j$ are disjoint if $i\neq j$. Pick a support plane $P_i$ at $p\in L_i$, and then we declare that $(\rho(g_0),\varrho(g_0))\cdot P_i$ is the chosen support plane at every point of $(\rho(g_0),\varrho(g_0))\cdot L_i$. This choice is well-defined by the above claim, which showed that if $(\rho(g),\varrho(g))$ leaves $L_i$ invariant, then it also leaves every support plane at $L_i$ invariant. Moreover this choice of support planes is invariant by the action of $\pi_1(S)$ by construction. This concludes the proof.
\end{proof}

\section*{Appendix: two lemmas in the hyperbolic plane}

We provide here the proofs of two properties on the action on $\Hyp^2\cup \partial_\infty\Hyp^2$ of sequences of elements in $\PSL(2,\R)$. We prove them by elementary arguments in the specific case of sequences of order--two elliptic isometries.

The first elementary property that we prove here is the uniform convergence of the action of elliptic isometries on the compactification of $\Hyp^2$.

\begin{lemma}\label{lemma:unif cont}
Let $w_n$ be a sequence in $\Hyp^2$ converging to $w\in \Hyp^2$. Then $\mathcal R_{w_n}$ converges to $\mathcal R_w$ uniformly on $\Hyp^2\cup\partial_\infty\Hyp^2$.
\end{lemma}
\begin{proof}
 Up to conjugation, we may assume $w=i$. Writing $w_n=\vert w_n \vert e^{i\eta_n}$, it is easy to check that 
 $$\mathcal R_{w_n}(z)=\frac{\cos(\eta_n)z-\vert w_n\vert}{\vert w_n \vert^{-1}z-\cos(\eta_n)}~.$$
Let us conjugate $\mathcal R_{w_n}$ by the map $\psi(z)=(iz+1)/(z+i)$, which maps $\Hyp^2$ to the disc, and show that it converges to $z\mapsto -z$ uniformly on the closed disc. For $z\in \Hyp^2\cup\partial_\infty\Hyp^2$ we have
$$\psi\circ \mathcal R_{w_n}\circ \psi^{-1}(z)+z=\frac{(\vert w_n \vert^{-1}-\vert w_n \vert-2i\cos(\eta_n))z^2+(\vert w_n \vert^{-1}-\vert w_n \vert+2i\cos(\eta_n))}{(\vert w_n \vert^{-1}-\vert w_n \vert-2i\cos(\eta_n))z+i(\vert w_n\vert+\vert w_n \vert^{-1})}$$
Hence  
$$\vert \psi\circ \mathcal R_{w_n}\circ \psi^{-1}(z)+z \vert \leq \frac{2\vert \alpha_n\vert}{\vert\alpha_n z+\beta_n\vert}$$
where $\alpha_n=\vert \vert w_n \vert^{-1}-\vert w_n \vert-2i\cos(\eta_n)\vert$ and $\beta_n=i(\vert w_n\vert+\vert w_n \vert^{-1})$.  Thus $$\vert \psi\circ \mathcal R_{w_n}\circ \psi^{-1}(z)+z \vert \leq \frac{2}{\vert z+\frac{\beta_n}{\alpha_n}\vert}\leq \frac{2}{\vert \vert \frac{\beta_n}{\alpha_n}\vert -\vert z \vert \vert}$$
Since $\vert \beta_n\vert\geq 2$, $\vert \alpha_n\vert\to 0$ and $\vert z\vert\leq 1$, there exists $n_0\in \mathbb{N}$ such that 
the right-hand side is smaller than $\epsilon$ for all $z$ in the closed disc.
This completes the proof.
\end{proof}

The second property is a special case of the so-called \emph{North-South dynamics}. 

\begin{lemma}\label{lemma:div rotations}
Let $w_n$ be a sequence in $\Hyp^2$ converging to $w\in \partial_\infty\Hyp^2$. Then, for every neighbourhood $U$ of $w$, there exists $n_0$ such that 
$\mathcal R_{w_n}((\Hyp^2\cup\partial_\infty\Hyp^2)\setminus U)\subset U$ for $n\geq n_0$.
\end{lemma}
\begin{proof}
We adopt the same notation as in the proof of Lemma \ref{lemma:unif cont}. Up to conjugation,  we may assume that $w=\infty$. It is sufficient to consider neighbourhoods $U$ of the form  $U_r=\{\vert z\vert>r \}\subset\Hyp^2\cup\partial_\infty\Hyp^2$. By a direct computation, 
$$ \vert \mathcal R_{w_n}(z)\vert=\frac{\vert\cos(\eta_n)z-\vert w_n\vert\vert}{\vert\vert w_n \vert^{-1}z-\cos(\eta_n)\vert} \geq \frac{\vert w_n\vert- \vert\cos(\eta_n)\vert\vert z\vert}{\vert w_n \vert^{-1}\vert z\vert+\vert \cos \eta_n\vert}.$$
Since $w_n$ converges to $\infty$, for all $r$ we have $\vert w_n\vert\geq r\geq\vert z\vert\geq \vert \cos \eta_n\vert \vert z\vert$ if $n$ is sufficiently large and $z$ is in the complement of $U_r$. Then
$$\vert \mathcal R_{w_n}(z)\vert\geq \frac{\vert w_n\vert-r}{\vert w_n \vert^{-1}{r}+\vert \cos \eta_n\vert}\longrightarrow +\infty.$$
It follows that $\vert \mathcal R_{w_n}(z)\vert>r$ for  $n\geq n_0$, that is, $\mathcal R_{w_n}$ maps the complement of $U_r$ to $U_r$.
\end{proof}

\end{document}